\documentclass[11pt,leqno]{article}

\usepackage{amsthm,amsfonts,amssymb,amsmath}
\usepackage{fullpage}
\usepackage{graphicx}
\usepackage{mathrsfs}
\usepackage{xcolor}
\numberwithin{equation}{section}
\usepackage{enumitem}
\usepackage{hyperref}
\usepackage{subcaption}
\usepackage[nocompress]{cite}

\makeatletter
\def\namedlabel#1#2{\begingroup
    #2%
    \def\@currentlabel{#2}%
    \phantomsection\label{#1}\endgroup
}
\makeatother

\newtheorem{Theorem}{Theorem}
\newtheorem{Lemma}{Lemma}
\newtheorem{Proposition}{Proposition}	
	
\theoremstyle{definition}
\newtheorem{Remark}{Remark}	 
\newtheorem{Definition}{Definition} 
\numberwithin{equation}{section}

\newcommand{\C}{\mathbb{C}} % komplexe
\newcommand{\R}{\mathbb{R}} % reelle
\newcommand{\N}{\mathbb{N}} % natuerliche
\renewcommand{\Re}{\operatorname{Re}}
\renewcommand{\Im}{\operatorname{Im}}
\newcommand{\iu}{\mathrm{i}}
\newcommand{\eu}{\mathrm{e}}
\newcommand{\eps}{\varepsilon}

\newcommand{\ub}{\mathbf{u}}

\newcommand{\ubu}{\underline{\mathbf{u}}}
\newcommand{\El}{\mathcal{L}}
\newcommand{\per}{\textup{per}}

\DeclareMathOperator{\spann}{span}

\DeclareMathOperator{\sech}{sech}

\newcommand\un[1]{\underline{#1}}

\allowdisplaybreaks[3]

\usepackage{caption}

\title{Existence and stability of soliton-based frequency combs in the Lugiato-Lefever equation}
\author{Lukas Bengel$^*$ and Bj\"orn de Rijk\thanks{Department of Mathematics, Karlsruhe Institute of Technology, Englerstra\ss e 2, 76131 Karlsruhe, Germany; \texttt{lukas.bengel@kit.edu}, \texttt{bjoern.de-rijk@kit.edu}}}

\begin{document}
\maketitle

\begin{abstract} 
Kerr frequency combs are optical signals consisting of a multitude of equally spaced excited modes in frequency space. They are generated by converting a continuous-wave pump laser within an optical microresonator. It has been observed that the interplay of Kerr nonlinearity and dispersion in the microresonator can lead to a stable optical signal consisting of a periodic sequence of highly localized ultra-short pulses, resulting in broad frequency spectrum. The discovery that stable broadband frequency combs can be generated in microresonators has unlocked a wide range of promising applications, particularly in optical communications, spectroscopy and frequency metrology. In its simplest form, the physics in the microresonator is modeled by the Lugiato-Lefever equation, a damped nonlinear Schr\"odinger equation with forcing. In this paper we demonstrate that the Lugiato-Lefever equation indeed supports arbitrarily broad Kerr frequency combs by proving the existence and stability of periodic solutions consisting of any number of well-separated, strongly localized and highly nonlinear pulses on a single periodicity interval. We realize these periodic multiple pulse solutions as concatenations of individual bright cavity solitons by phrasing the problem as a reversible dynamical system and employing results from homoclinic bifurcation theory. The spatial dynamics formulation enables us to harness general results, based on Evans-function techniques and Lin's method, to rigorously establish diffusive spectral stability. This, in turn, yields nonlinear stability of the periodic multipulse solutions against localized and subharmonic perturbations.

\paragraph*{Keywords.} Lugiato-Lefever equation, periodic multipulse solutions, spectral analysis, nonlinear stability, bifurcation theory\\
\textbf{Mathematics Subject Classification (2020).} Primary, 35B10, 35Q55, 35B35; Secondary, 35C08, 34C23
\end{abstract}

% 35B10 periodic solutions to PDEs
% 35C08 Soliton solutions
% 35B35 Stability in context of PDEs
% 35Q55 NLS equations (nonlinear Schroedinger equations)
% 34C23 Bifurcation theory for ordinary differential equations

% Start

\section{Introduction}

In this paper we rigorously construct periodic multipulse solutions to the Lugiato-Lefever equation and determine their spectral and nonlinear stability. The Lugiato-Lefever equation (LLE) is a damped and forced nonlinear Schr\"odinger equation given by
\begin{align} \label{LLE_intro}
    \iu u_t = -d u_{xx} + \zeta u - |u|^2 u - \iu u + \iu f,
\end{align}
where $u\colon\R\times\R \to \C$ is a complex-valued function, $d \neq 0$ denotes the dispersion, $\zeta \in \R$ is a detuning parameter and $f > 0$ represents the forcing. The LLE was derived from Maxwell's equations in~\cite{Lugiato1987Spatial} to describe the optical field in a dissipative and nonlinear cavity filled with a Kerr medium and subjected to a continuous-wave laser pump. As such it serves as a canonical model, see~\cite{Chembo2017} and further references therein, for Kerr frequency comb generation in continuous-wave laser driven microresonators, which are microscopic ring- or disk-shaped cavities that confine light by circulating it in a closed path and enhance the interaction of light through resonance. Kerr frequency combs, which are optical signals whose frequency spectrum is carved into a series of regularly spaced $\delta$-functions, arise due to four-wave mixing mediated by the nonlinear Kerr effect in the microresonator. 

Over the past decades the generation of combs with broad frequency spectrum has become the subject of intensive research, mainly due to the fact that such broad bandwidth frequency combs have revolutionized the precision and accuracy with which different optical transition frequencies can be measured, a discovery that was awarded with the Nobel prize in physics and has promising applications to optical communications~\cite{Marin-Palomo2017}, broadband gas sensing~\cite{Schliesser2005}, spectroscopy~\cite{Coddington2008,Picque2019}, and frequency metrology~\cite{Udem2002}, to name but a few. The generation of broadband frequency combs in high-quality microresonators has sparked significant interest~\cite{DelHaye2007,Ferdous2011}, mainly due to the potential of chip-scale implementation, which facilitates the integration of frequency comb technology into applications outside the laboratory. As high sensitivity to noise is undesired for practical implementation, attention must be given to the stability of these combs. 

A breakthrough addressing the stability issue is the experimental realization~\cite{Brasch2016Photonic,Herr2014Temporal} of frequency combs comprised of a multitude of well-separated bright (cavity) solitons, which are remarkably stable thanks to a double balance between anomalous dispersion and the Kerr nonlinearity (which defines their shape) and between gain and dissipation (which defines their amplitude). The individual solitons correspond to ultrashort pulses whose frequency spectrum is broad and smooth, which, together with their stability properties, makes soliton-based frequency combs highly attractive for applications, see e.g.~\cite{Marin-Palomo2017,Weiner2017} and further references therein.

In experiments~\cite{Brasch2016Photonic,Herr2014Temporal} the number of solitons constituting the generated frequency comb turns out to be stochastic, but, once generated, the waveform is stable. That is, on a single periodicity interval $I \subset \R$, provided by the ring (or disk) shape of the microresonator, stable optical signals of any number $N \in \mathbb{N}$ of ultrashort, well-separated, pulses can be generated. In simulations in~\cite{Herr2014Temporal} of the microresonator system with parameters similar to the experimental setup a close-to-perfect match was found between the numerical solution and the formal approximate solution
\begin{align} \label{formal}
u(x) = \alpha_{CW} + \alpha \sum_{i = 1}^N \phi_{\theta}(x - X_i),
\end{align}
corresponding to $N \in \mathbb{N}$ solitons superposed on a background $\alpha_{CW} \in \C$. Here, $x \in I$ resembles the angular coordinate inside the resonator, $\alpha \in \C \setminus \{0\}$ denotes the amplitude, $X_i \in I$ represents the position of the $i$-th soliton and 
\begin{align} \label{bright_sol}
\phi_{\theta}(x) = \sqrt{2\zeta} \sech\left(\sqrt{\zeta}\, x\right) \eu^{\iu \theta}
\end{align}
is the well-known bright soliton with phase $\theta \in \R$ solving the focusing nonlinear Schr\"odinger (NLS) equation
\begin{align} \label{NLS}
    \iu u_t = -u_{xx} + \zeta u - |u|^2 u
\end{align}
with detuning parameter $\zeta > 0$. 

Subsequent bifurcation analyses~\cite{Godey2014Stability,ParraRivas2018BifurcationLocalized,Gaertner2019Bandwidth,Mandel2017Apriori} based on numerical continuation indicate that frequency combs consisting of a multitude of bright solitons on a periodicity interval can be found in the LLE~\eqref{LLE_intro} in the anomalous dispersion regime $d > 0$. More precisely, these numerical analyses suggest that, as a result of a snaking bifurcation, the LLE supports periodic pulse solutions comprised of any number of solitons on a single periodicity interval. 

\subsection{Main result}

In this paper we affirm the above experimental and numerical findings by proving that the LLE~\eqref{LLE_intro} supports stable soliton-based frequency combs, whose leading-order profiles are of the form~\eqref{formal} on a single periodicity interval. These periodic $N$-pulse solutions to~\eqref{LLE_intro} arise in the anomalous dispersion regime $d > 0$ with small damping and forcing. In this regime stable $1$-pulse solutions to~\eqref{LLE_intro} on $\R$ were constructed in~\cite{Bengel2024Stability} by bifurcating from the rotated bright soliton solution~\eqref{bright_sol} to the focusing NLS equation~\eqref{NLS}. Thus, we regard the LLE as a perturbed focusing NLS equation
\begin{align}\label{LLE}
    \iu u_t = -u_{xx} + \zeta  u - |u|^2 u + \eps\iu( -u+ f)
\end{align}
with parameters $\zeta, f > 0$ and $0 < \eps \ll 1$, where we have set the dispersion to $1$ by rescaling space. We note that, given a dispersion coefficient $d > 0$, solutions of~\eqref{LLE} are in $1$-to-$1$-correspondence with solutions of the original formulation~\eqref{LLE_intro} of the LLE, see Remark~\ref{rem:scalings}. 

Our main result may now be formulated as follows.

\begin{Theorem} \label{t:mainresult}
Fix $N \in \mathbb{N}$. Set $n = {\lfloor \frac{N}{2} \rfloor}$ and $\alpha_0 = N\! \mod 2 \in \{0,1\}$. Assume that $\zeta, f > 0$ and $\theta_0 \in \R$ obey $8 \zeta < \pi^2 f^2$, $\pi f \cos \theta_0 = 2 \sqrt{2 \zeta}$ and $\sin \theta_0 > 0$. Then, there exist constants $C_0,\eps_0 > 0$ such that for all $\eps \in (0,\eps_0)$, $\mathbf{k} = (k_1,\ldots,k_n) \in \mathbb{N}^n$ and $m \in \mathbb{N}$ there exist distances $\smash{T_{1,\eps}^{k_1}, \ldots, T_{n,\eps}^{k_n} > 0}$ and periods $L_{\mathbf{k}}^m > 0$ satisfying
\begin{align} \label{sum_period}
2 \sum_{i = 1}^n T_{i,\eps}^{k_i} < L_{\mathbf{k}}^m
\end{align}
such that equation~\eqref{LLE} admits a stationary smooth solution $u_{m,\mathbf{k},\eps} \colon \R \to \C$ enjoying the following properties:
\begin{itemize}
\item[(i)] \emph{\bf (Symmetry).} The solution $u_{m,\mathbf{k},\eps}$ is even, i.e., it holds $u_{m,\mathbf{k},\eps}(x) = u_{m,\mathbf{k},\eps}(-x)$ for all $x \in \R$, $\eps \in (0,\eps_0)$, $\mathbf{k} \in \N^n$ and $m \in \mathbb{N}$. 
\item[(ii)] \emph{\bf (Periodicity).} The solution $u_{m,\mathbf{k},\eps}$ is $\smash{L_{\mathbf{k},\eps}^m}$-periodic for each $\eps \in (0,\eps_0)$, $\mathbf{k} \in \N^n$ and $m \in \mathbb{N}$. For each fixed $\eps \in (0,\eps_0)$ and $\mathbf{k} \in \N^n$ the sequence $\smash{\{L_{\mathbf{k},\eps}^m\}_m}$ of periods is monotonically increasing and tends to $\infty$ as $m \to \infty$. 
\item[(iii)] \emph{\bf (Approximation).} On a single periodicity interval the solution $u_{m,\mathbf{k},\eps}$ is approximated by a superposition of $N$ rotated bright solitons of the form~\eqref{bright_sol} as
\begin{align*}
\left|u_{m,\mathbf{k},\eps}(x) - \alpha_0 \phi_{\theta_0}(x) - \sum_{i = 1}^n \left(\phi_{\theta_0}\left(x - T_{1,\eps}^{k_1} - \ldots - T_{i,\eps}^{k_i}\right) + \phi_{\theta_0}\left(x + T_{1,\eps}^{k_1} + \ldots + T_{i,\eps}^{k_i}\right)\right)\right| \leq C_0 \eps
\end{align*}
for $x \in \smash{[-\frac12 L_{\mathbf{k},\eps}^m,\frac12 L_{\mathbf{k},\eps}^m]}$, $\eps \in (0,\eps_0)$, $\mathbf{k} \in \N^n$ and $m \in \mathbb{N}$.
\item[(iv)] \emph{\bf (Pulse separation).} For all $i = 1,\ldots,n$ and $\eps \in (0,\eps_0)$ the sequence $\{T_{i,\eps}^{k}\}_k$ of pulse distances is monotonically increasing  with $T_{i,\eps}^k \to \infty$ as $k \to \infty$.
\item[(v)] \emph{\bf (Asymptotic orbital stability against subharmonic perturbations).} Let $\eps \in (0,\eps_0)$, $\mathbf{k} \in \N^n$ and $m, M \in \mathbb{N}$. There exist constants $C,\delta,\eta > 0$ such that for all $v_0 \in \smash{H^1_{\mathrm{per}}(0,ML_{\mathbf{k},\eps}^m)}$ with $\smash{\|v_0\|_{H^1_{\mathrm{per}}(0,ML_{\mathbf{k},\eps}^m)}} < \delta$ there exist a constant $\gamma \in \R$ and a global (mild) solution 
\begin{align*}u \in C\big([0,\infty),H^1_{\mathrm{per}}(0,ML_{\mathbf{k},\eps}^m)\big)\end{align*} 
of~\eqref{LLE} with initial condition $u(0) = u_{m,\mathbf{k},\eps} + v_0$ satisfying
\begin{align*}
|\gamma| \leq C\|v_0\|_{H^1_{\mathrm{per}}(0,ML_{\mathbf{k},\eps}^m)} , \qquad \|u(\cdot,t) - u_{m,\mathbf{k},\eps}(\cdot + \gamma)\|_{H^1_{\mathrm{per}}(0,ML_{\mathbf{k},\eps}^m)} \leq C \eu^{-\eta t} \|v_0\|_{H^1_{\mathrm{per}}}
\end{align*}
for $t \geq 0$.
\item[(vi)] \emph{\bf (Diffusive stability against localized perturbations).} Let $\eps \in (0,\eps_0)$ $\mathbf{k} \in \N^n$ and $m \in \mathbb{N}$. There exist constants $C,\delta> 0$ such that for all $v_0 \in L^1(\R) \cap H^4(\R)$ with $\|v_0\|_{L^1 \cap H^4} < \delta$ there exist functions 
\begin{align*}\gamma, v \in C\big([0,\infty),H^4(\R)\big) \cap C^1\big([0,\infty),H^2(\R)\big)\end{align*} 
satisfying $\gamma(0) = 0$ and $v(0) = v_0$ such that $u = u_{m,\mathbf{k},\eps} + v$ is the unique global classical solution of~\eqref{LLE} with $u(0) = u_{m,\mathbf{k},\eps} + v_0$. Moreover, the estimates
\begin{align*}
\|\gamma(t)\|_{L^2}, \|u(t) - u_{m,\mathbf{k},\eps}\|_{L^2} &\leq C (1+t)^{-\frac14} \|v_0\|_{L^1 \cap H^4},\\
\|u(\cdot,t) - u_{m,\mathbf{k},\eps}(\cdot + \gamma(\cdot,t))\|_{L^2} &\leq C (1+t)^{-\frac34} \|v_0\|_{L^1 \cap H^4}
\end{align*}
hold for $t \geq 0$.
\end{itemize}
\end{Theorem}

Theorem~\ref{t:mainresult} shows that for any $N \in \mathbb{N}$ there exist even periodic stationary solutions to~\eqref{LLE} composed of a superposition of $N$ bright solitons on a single periodicity interval. The period length, as well as the distance between the individual solitons, can be chosen arbitrarily large through the $n+1$ degrees of freedom $k_1,\ldots,k_n$ and $m$ in Theorem~\ref{t:mainresult}, where we set $\smash{n = \lfloor \frac{N}{2} \rfloor}$. More precisely, after fixing $k_1,\ldots,k_i \in \N$ so that the distances $T_{1,\eps}^{k_1},\ldots,T_{i,\eps}^{k_i}$ between the first $i \in \{1,\ldots,n\}$ solitons and their symmetric counterparts are set, there remain $n+1-i$ degrees of freedom, namely $k_{i+1},\ldots,k_n, m\in \N$, which can still be adjusted to make the distances between the remaining $N-2i$ solitons, as well as the period of the solution, arbitrarily large. In particular, for fixed parameters $\zeta,f > 0$, $N \in \mathbb{N}$ and $\eps > 0$, Theorem~\ref{t:mainresult} yields a $(n+1)$-parameter family of even periodic stationary $N$-pulse solutions of~\eqref{LLE}, whose pulse locations can be unequally spaced. Since any fixed spatial translate of a solution of~\eqref{LLE} is again a solution, we find that Theorem~\ref{t:mainresult} provides in fact a $(n+2)$-parameter family of periodic stationary solutions to~\eqref{LLE}.

Upon rescaling the spatial variable $x$, the optical field $u$, and the detuning and forcing parameters $\zeta,f > 0$ in~\eqref{LLE} and upon reintroducing the dispersion parameter $d > 0$, we can adjust the period length to match the (fixed) circumference of the microresonator and we can normalize the damping coefficient, see Remark~\ref{rem:scalings}. In particular, one finds after rescaling that the periodic $N$-pulse solutions in Theorem~\ref{t:mainresult} correspond to frequency comb solutions of~\eqref{LLE_intro} in the regime of large detuning, strong forcing and high anomalous dispersion. These frequency combs are comprised of ultrashort, well-separated, large-amplitude bright solitons and were thus numerically and experimentally observed in~\cite{Brasch2016Photonic,Herr2014Temporal,Godey2014Stability,ParraRivas2018BifurcationLocalized,Gaertner2019Bandwidth,Mandel2017Apriori}. In fact, the amplitude of the frequency comb solution to~\eqref{LLE_intro} can be made arbitrarily large by taking $\eps > 0$ sufficiently small in Theorem~\ref{t:mainresult}, cf.~Remark~\ref{rem:scalings}. On the other hand,  by fixing $\eps > 0$ and taking the parameters $k_1,\ldots,k_n$ and $m$ in Theorem~\ref{t:mainresult} sufficiently large, the individual solitons constituting the frequency comb can be arbitrarily localized, while maintaining the same amplitude. Therefore, their frequency spectrum can be made arbitrarily broad. 

The periodic multipulse solutions, or \emph{multipulse trains}, established in Theorem~\ref{t:mainresult}, exhibit the strongest stability properties attainable for (nonconstant) stationary periodic solutions of~\eqref{LLE}. Since any spatial translate corresponds to a co-periodic perturbation, $u_{m,\mathbf{k},\eps}$ cannot be asymptotically stable against subharmonic perturbations. However, Theorem~\ref{t:mainresult} shows that, except for a small spatial shift of the original periodic solution, the effect of subharmonic perturbations fades exponentially quickly in time. In addition, by Floquet-Bloch theory, cf.~\cite{Haragus2021Linear,Gardner1993Structure}, the linearization of~\eqref{LLE} about $u_{m,\mathbf{k},\eps}$ possesses, when posed on $L^2(\R)$, continuous spectrum, which, in the most stable situation, touches the origin due to translational invariance in a single quadratic tangency. It is well-known that such \emph{diffusively stable} spectrum, cf.~Definition~\ref{def:diffusive_spectral_stability}, leads to algebraic decay rates of perturbations, whose leading-order behavior is captured by a diffusively decaying spatio-temporal phase modulation satisfying a viscous Hamilton-Jacobi equation, cf.~\cite{Haragus2023Nonlinear,Zumbrun2024Forward,Sandstede2012,Doelman2009}. That is, the algebraic decay rates stated in Theorem~\ref{t:mainresult} are the best achievable in the case of localized perturbations. 

In summary, Theorem~\ref{t:mainresult} rigorously shows that the LLE admits for any number $N \in \mathbb{N}$ an $(\lfloor \frac{N}{2} \rfloor + 2)$-parameter family of frequency combs. These combs are periodic solutions of~\eqref{LLE_intro} consisting of $N$ well-separated, generally unequally spaced bright solitons on a single periodicity interval. The amplitude and frequency spectra of these solitions can be made arbitrarily large and broad, respectively. Moreover, the frequency combs exhibit the best attainable stability properties. These features are, as outlined above, of key importance for applications relying on frequency-comb technology. Moreover, as explained in~\S\ref{sec:embedding} below, Theorem~\ref{t:mainresult} is the first rigorous mathematical result establishing periodic \emph{multiple} pulse solutions of the LLE. In addition, the solutions in Theorem~\ref{t:mainresult} are the first far-from-equilibrium periodic solutions of the LLE, whose stability against localized perturbations and against subharmonic perturbations of any wavelength has been rigorously proven.

\begin{Remark} \label{rem:scalings}
Let $d > 0$. If $u(x,t)$ is a solution to~\eqref{LLE}, then the rescaled solution $\tilde{u}(x,t) = \eps^{-1/2} u((d\eps)^{-1/2}x,\eps^{-1} t)$ solves
\begin{align*}
\iu \tilde{u}_t = -d\tilde{u}_{xx} + \tilde{\zeta} \tilde{u} - |\tilde{u}|^2 \tilde{u} - \iu \tilde{u} + \iu \tilde{f}
\end{align*}
with $\tilde{\zeta} =\eps^{-1}\zeta$ and $\tilde{f} = \eps^{-1/2} f$. Hence, we find that solutions of~\eqref{LLE} are in $1$-to-$1$-correspondence with solutions of the original formulation~\eqref{LLE_intro} of the LLE with normalized damping coefficient and dispersion coefficient $d > 0$. In particular, the solutions $u_{m,\mathbf{k},\eps}$, established in Theorem~\ref{t:mainresult}, correspond to stable pulse train solutions of~\eqref{LLE_intro} of period $(d \eps)^{1/2} L_{\mathbf{k},\eps}^m$, whose amplitude can be made arbitrarily large by taking $\eps > 0$ sufficiently small. Moreover, choosing $d = 1/(\eps L_{\mathbf{k},\eps}^m)^{1/2}$ yields $1$-periodic multipulse solutions to~\eqref{LLE_intro}, whose individual pulses become highly localized and well-separated by taking $k_1,\ldots,k_n, m \in \mathbb{N}$ sufficiently large.
\end{Remark}

\subsection{Embedding in the mathematical literature} \label{sec:embedding}

Despite the importance of frequency combs to applications, there are relatively few mathematically rigorous studies of periodic solutions to the LLE. The first mathematical works~\cite{Miyaji2010Bifurcation,Delcey2018Periodic,Delcey2018Instability,Godey2017Bifurcation} focus on proving the existence of small amplitude periodic solutions of~\eqref{LLE_intro} by bifurcating from spatially homogeneous steady states. 
These stationary periodic solutions are weakly nonlinear patterns, i.e., their leading-order profile is a (co)sine wave superposed on the homogeneous background state. In particular, they do not exhibit broad frequency spectrum.

To date, far-from-equilibrium periodic solutions have, to the authors' best knowledge, only been established rigorously in~\cite{Mandel2017Apriori,Hakkaev2019Generation}. In~\cite{Mandel2017Apriori} global branches of stationary periodic solutions and bounds on their location in parameter space were obtained using global bifurcation theory. Yet, the results in~\cite{Mandel2017Apriori} do not provide any rigorous mathematical control on the profile or size of the periodic solutions away from the branch of homogeneous background states. 

In~\cite{Hakkaev2019Generation} stationary periodic solutions to~\eqref{LLE_intro} are constructed by bifurcating from the well-known one-parameter family of real-valued periodic dnoidal solutions of the focusing NLS equation~\eqref{NLS}. The homoclinic limit of the real dnoidal family is the bright soliton~\eqref{bright_sol} with $\theta \in \{0,\pi\}$. As the dnoidal waves approach the homoclinic limit, their profile thus consists of one strongly localized bright soliton on each periodicity interval. Consequently, the bifurcating solutions to the LLE resemble periodic $1$-pulse solutions, whose individual pulses exhibit broad frequency spectrum. These bifurcating periodic $1$-pulse solutions are however different from the ones constructed in this paper, because they are unstable against localized and long-wavelength subharmonic perturbations, whereas the pulse trains in Theorem~\ref{t:mainresult} are stable against localized perturbations and against subharmonic perturbations of any wavelength. We refer to Remark~\ref{rem:dnoidal} for further details.

As far as the authors are aware, the only class of periodic solutions to~\eqref{LLE_intro} whose stability against localized perturbations and against subharmonic perturbations of any wavelength has been rigorously established in the current literature~\cite{Haragus2023Nonlinear,Stanislavova2018Asymptotic}, are the small-amplitude, weakly nonlinear Turing patterns constructed in~\cite{Delcey2018Periodic,Delcey2018Instability}. This follows by the fact that they are diffusively spectrally stable, as proved in~\cite{Delcey2018Instability,Delcey2018Periodic}. On the other hand, spectral and nonlinear stability of small amplitude periodic solutions of~\eqref{LLE_intro} against co-periodic perturbations is obtained in~\cite{Miyaji2010Bifurcation,Miyaji2011Stability}. Finally, although the periodic solutions to~\eqref{LLE_intro} bifurcating from the family of dnoidal NLS-solutions are unstable against localized perturbations and large-wavelength subharmonic perturbations, their spectral and nonlinear stability against co-periodic perturbations is proven in~\cite{Hakkaev2019Generation,Stanislavova2018Asymptotic}, thereby confirming the formal asymptotic analysis presented in~\cite{Sun2018}.

\begin{Remark} \label{rem:dnoidal}

The one-parameter family of real-valued periodic dnoidal solutions of the focusing NLS equation~\eqref{NLS} can be extended to a two-parameter family of stationary periodic solutions through rotation. Indeed, if $u$ is a stationary solution of~\eqref{NLS}, so is $\eu^{\iu \theta} u$ for any $\theta \in \R$. For each fixed rotation angle $\theta \in \R$ the homoclinic limit of the family is given by the rotated bright solition~\eqref{bright_sol}.  

It is a classical result that any nonconstant real-valued stationary solution of period $T > 0$ of the focusing NLS equation~\eqref{NLS} is long-wavelength (or sideband) unstable~\cite{Rowlands1974Stability}, i.e., it is spectrally unstable against $MT$-periodic perturbations for $M \in \N$ sufficiently large. In particular, any real-valued dnoidal solution of~\eqref{NLS} is long-wavelength unstable. Since the spectrum is unaffected by the rotation $u \mapsto \eu^{\iu \theta} u$, it follows that the full two-parameter family of dnoidal waves is sideband unstable. 

The long-wavelength instability is inherited by periodic solutions of the LLE bifurcating from any member of the two-parameter dnoidal family. The reason is as follows. The linearization of~\eqref{LLE_intro} or~\eqref{NLS} about a stationary $T$-periodic solution posed on $L^2_{\mathrm{per}}(0,MT)$ has compact resolvent for any $M \in \mathbb{N}$ due to the compact embedding of its domain $H^2_{\mathrm{per}}(0,MT)$ into $L^2_{\mathrm{per}}(0,MT)$ by the Rellich–Kondrachov theorem. Consequently, its spectrum consists of isolated eigenvalues of finite multiplicity. It is well-known, cf.~\cite[Section~4.3.5]{Kato1995Perturbation}, that a finite set of eigenvalues of finite multiplicity changes continuously under bounded perturbations. Therefore, stationary $T$-periodic solutions to~\eqref{LLE_intro} bifurcating from any $T$-periodic dnoidal solution of~\eqref{NLS} are long-wavelength unstable. By Floquet-Bloch theory, cf.~\cite{Gardner1993Structure,Haragus2021Linear}, the spectrum of the linearization of~\eqref{LLE_intro} about a $T$-periodic stationary wave posed on $L^2(\R)$ arises by taking the union over $M \in \N$ of all spectra of the linearizations of~\eqref{LLE_intro} about the wave posed on $L^2_{\mathrm{per}}(0,MT)$. Hence, any long-wavelength unstable periodic stationary solution is also spectrally unstable against localized perturbations. We conclude that the periodic stationary solutions of the LLE, which were established in~\cite{Hakkaev2019Generation} by bifurcating from the dnoidal waves, are (spectrally) unstable against localized perturbations and against $MT$-periodic perturbations for $M \in \mathbb{N}$ sufficiently large. This is confirmed by numerical simulations in~\cite{Sun2018}. Interestingly, these simulations also indicate that, outside of a neighborhood of the bifurcation point, the $L^2_{\mathrm{per}}(0,MT)$-spectrum might stabilize for given $M \in \mathbb{N}$.

However, the homoclinic limits, given by the rotated bright solitons~\eqref{bright_sol}, of the two-parameter dnoidal family are spectrally and orbitally stable~\cite{Cazenave1982Orbital,Weinstein1985Modulational,Weinstein1986Lyapunov} as solutions to the NLS equation~\eqref{NLS} against localized perturbations. It has been shown in~\cite{Bengel2024Stability} that the spectral stability is inherited by some of the bifurcating soliton solutions of~\eqref{LLE}, see also Theorem~\ref{thm:1-pulse}. In this paper, we exploit the spectral stability of the soliton solutions to construct periodic multipulse solutions, which are spectrally and nonlinearly stable against localized perturbations and against subharmonic perturbations of any wavelength.
\end{Remark}

\subsection{Dynamical systems approach} \label{sec:approach}

The results presented in this paper are the outcome of a dynamical systems approach to analyze the existence and spectral stability problems for stationary solutions to~\eqref{LLE}. These problems, being independent of time, may be written as first-order dynamical systems of ordinary differential equations in the spatial variable $x$. Due to the reflection symmetry $x \mapsto -x$ present in~\eqref{LLE}, these dynamical systems admit a reversible symmetry.

Our basic ingredients for the construction of periodic multipulse solution to~\eqref{LLE} are the $1$-pulse solutions bifurcating from the one-parameter family of rotated bright NLS solitons~\eqref{bright_sol}. These primary $1$-pulse solutions to~\eqref{LLE} were rigorously established in~\cite{Gaertner2020Soliton,Gasmi2022Lugiato,Bengel2024Stability} and their spectral and nonlinear stability was analyzed in~\cite{Bengel2024Stability} using Krein index counting and analytic perturbation theory. Upon formulating the existence problem as a reversible dynamical system, the $1$-pulses correspond to symmetric nondegenerate homoclinics to a saddle-focus equilibrium. Homoclinic bifurcations results~\cite{Sandstede1999Instability,Haerterich1998Cascades,Champneys1994Homoclinic} for reversible dynamical systems, which rely on Shil’nikov analysis or a Lyapunov-Schmidt reduction method called Lin's method~\cite{Lin1990Melnikovs,Sandstede1993Verzweigungstheorie}, allow us to concatenate any number $N \in \mathbb{N}$ of these nondegenerate homoclinics, yielding so-called \emph{$N$-homoclinics}, which are again nondegenerate and symmetric. The $N$-homoclinics correspond to even stationary multiple pulse solutions of~\eqref{LLE} comprised of $N$ well-separated pulses, which can, in turn, be approximated by the rotated bright NLS solitons~\eqref{bright_sol}. The spectral stability of these $N$-pulse solutions follows by combining results from~\cite{Sandstede1998Stability,Sandstede1999Instability,BengeldeRijk} with a-priori bounds on the spectrum. More specifically, in~\cite{Sandstede1998Stability} general eigenvalue problems arising in the spectral stability analysis of $N$-pulses bifurcating from a formal concatenation of $N$ primary pulses are studied using Lin's method, providing leading-order control over the $N$ small eigenvalues bifurcating from the translational eigenvalue residing at the origin. An application of the theory of~\cite{Sandstede1998Stability} to reversible systems can be found in~\cite{Sandstede1999Instability}, see also~\cite{Sandstede1997Fibers}. On the other hand, the Evans-function analysis in~\cite{BengeldeRijk} yields that the absence of eigenvalues in compact regions of the spectral plane associated with the primary $1$-pulse solutions is inherited by the bifurcating $N$-pulse solutions. We note that the spectral stability of the multipulse solutions implies their asymptotic orbital stability through standard arguments relying on high-frequency resolvent bounds established in~\cite{Bengel2024Stability}. We emphasize that these rigorous existence and stability results of multiple pulse solutions to the LLE are novel and interesting in their own right. We refer to~\S\ref{ssec:ex_N-pulse} and~\S\ref{ssec:stab_N-pulse} for the precise statements.

Having established a nondegenerate $N$-homoclinic in the dynamical systems formulation of the existence problem, we employ the homoclinic bifurcation results in~\cite{Vanderbauwhede1992Homoclinic}, which again rely on Lin's method, to find nearby periodic orbits, which have large spatial periods $T > 0$ and are reversibly symmetric. The bifurcating periodic orbits correspond to a family of periodic $N$-pulse solutions to~\eqref{LLE}. We study the spectral stability of these $T$-periodic pulse solutions against localized perturbations and subharmonic perturbations using~\cite{Sandstede2001Stability,BengeldeRijk,Gardner1997Spectral}. There are $M$ eigenvalues of the linearization of~\eqref{LLE} about the multipulse train posed on $L^2_{\mathrm{per}}(0,MT)$ bifurcating from each isolated eigenvalue associated with the underlying multipulse solution, cf.~\cite{Gardner1993Structure}. Since there is an eigenvalue associated with the underlying multipulse solution at the origin due to translational invariance, the multipulse train can be spectrally unstable even in the case of spectral stability of the underlying multipulse. In the case of localized perturbations each eigenvalue associated with the underlying multipulse solution yields a bifurcating spectral curve consisting of the union of eigenvalues of the linearizations posed on $L^2_{\mathrm{per}}(0,MT)$ for each $M \in \N$. Leading-order control on the bifurcating eigenvalues in the case of subharmonic perturbations and on the bifurcating spectral curve in the case of localized perturbations close to the origin is provided by results in~\cite{Sandstede2001Stability}, which rely on Lin's method and Floquet-Bloch theory. On the other hand, the Evans function analysis in~\cite{BengeldeRijk,Gardner1997Spectral} yield that the absence of eigenvalues in compact regions of the spectral plane associated with the underlying multipulse solution is inherited by the bifurcating multipulse trains. Combining the latter with spectral a-priori bounds then leads to the desired diffusive spectral stability result for the periodic multipulse solution. Finally, diffusive spectral stability implies nonlinear stability against localized perturbations and against subharmonic perturbations of any wavelength, cf.~\cite{Haragus2023Nonlinear,Haragus2024Nonlinear,Stanislavova2018Asymptotic}.

\subsection{Outline of paper}

The remainder of this paper is structured as follows. In~\S\ref{sec:soliton-based} we collect previous results on the existence and spectral stability of the $1$-pulse solutions to the LLE bifurcating from the rotated bright NLS-solitons. Moreover, we formulate the existence problem as a dynamical system and establish multiple and periodic pulse solutions. The spectral and nonlinear stability analysis of the multiple and periodic pulse solutions, as well as the proof of our main result, Theorem~\ref{t:mainresult}, can be found in~\S\ref{sec:stability}. 

\paragraph*{Acknowledgments.}  This project is funded by the Deutsche Forschungsgemeinschaft (DFG, German Research Foundation) -- Project-ID 258734477 -- SFB 1173

\section{Soliton-based pulse solutions} \label{sec:soliton-based}

In this section we establish multiple and periodic pulse solutions to~\eqref{LLE}. As outlined in~\S\ref{sec:approach}, the fundamental building blocks of these pulse solutions are the stationary $1$-pulse solutions of~\eqref{LLE}, constructed in~\cite{Gasmi2022Lugiato,Gaertner2020Soliton,Bengel2024Stability} by bifurcating from the $1$-parameter family of rotated bright NLS solitions~\eqref{bright_sol}. After we have formulated the existence problem as a dynamical system, we collect the relevant properties of these primary $1$-pulse solutions from~\cite{Bengel2024Stability} and show that they correspond to nondegenerate symmetric homoclinics in the dynamical system. Then, we employ  homoclinic bifurcation
results from~\cite{Sandstede1993Verzweigungstheorie,Sandstede1999Instability,Vanderbauwhede1992Homoclinic} to obtain the desired multiple and periodic pulse solutions.

\subsection{Spatial dynamics formulation}

Multiplying~\eqref{LLE} by $-\iu$ and decomposing $u$ into its real and imaginary part yields the real two-component system
\begin{equation}\label{LLE_system}
    \ub_t = J (- \ub_{xx} + \zeta \ub - |\ub|^2\ub ) + \eps (-\ub + \mathbf{F}),
\end{equation}
written in vector notation $\ub = (\Re(u), \Im(u))^\top$, where we denote
\begin{align*}
    J = 
    \begin{pmatrix}
        0 & 1 \\ -1 & 0
    \end{pmatrix}, \qquad
    \mathbf{F} =
    \begin{pmatrix}
        f \\ 0
    \end{pmatrix}
\end{align*}
and where $|\ub| = \sqrt{\ub_1^2 + \ub_2^2}$ is the usual Euclidean norm. The advantage of the fomulation~\eqref{LLE_system} over~\eqref{LLE} is that the nonlinearity is now a (Fr\'echet) differentiable function of the vector $\ub$. We note that any real-valued solution $\ub = (\ub_1,\ub_2)^\top$ to~\eqref{LLE_system} gives rise to a complex-valued solution to~\eqref{LLE} by setting $u = \ub_1 + \iu \ub_2$.

Stationary solutions of~\eqref{LLE_system} solve the second-order system
\begin{align}\label{LLE_system_existence}
    J (- \ub_{xx} + \zeta \ub - (\ub_1^2+\ub_2^2)\ub ) +\eps ( -\ub + \mathbf{F}) = 0
\end{align}
of ordinary differential equations. By introducing the variable
\begin{align*}
    U = (\ub_1, \ub_2, \ub_1', \ub_2')^\top
\end{align*}
the existence problem~\eqref{LLE_system_existence} can be written as a dynamical system on $\R^4$ given by
\begin{align}\label{LLE_DS}
    U' = F(U;\eps),
\end{align}
where $F \colon \R^4 \times \R \to \R^4$ is the smooth nonlinear function defined by
\begin{align*}
     F(U;\eps) =
    \begin{pmatrix}
        U_3 \\ U_4 \\
        \zeta U_1 + \eps U_2 -(U_1^2 + U_2^2) U_1 \\
        \zeta U_2 - \eps U_1 -(U_1^2 + U_2^2) U_2 + \eps f
    \end{pmatrix}.
\end{align*}
The reflection symmetry $x \mapsto -x$ of~\eqref{LLE} yields that the first-order system \eqref{LLE_DS} of ordinary differential equations is \emph{reversible} for the linear involution
\begin{align*}
    R = 
    \begin{pmatrix}
        1  & 0 & 0 & 0 \\
        0 & 1 & 0 & 0 \\
        0 & 0 & -1 & 0 \\
        0 & 0 & 0 & -1
    \end{pmatrix} \in \R^{4 \times 4}, \qquad R^2 = R.
\end{align*}
That is, the relation $RF(U;\eps) = - F(RU;\eps)$ holds true for all $U \in \R^4$ and $\eps \in \R$. A solution $U$ of \eqref{LLE_DS} is called \emph{symmetric} if we have $RU(-x) = U(x)$ for all $x \in \R$. We note that symmetric solutions of~\eqref{LLE_DS} give rise to even stationary solutions of~\eqref{LLE}. Since $F$ is smooth, we find that all solutions of~\eqref{LLE_DS} are smooth by standard local existence theory for ODEs. Consequently, all stationary bounded solutions of~\eqref{LLE_system} are smooth. 

Taking advantage of the differentiability of the nonlinearity in~\eqref{LLE_system}, we can linearize system~\eqref{LLE_system} about a bounded stationary solution $\underline{\ub} = (\underline{\ub}_1,\underline{\ub}_2)^\top \colon \R \to \R^2$. The linearization of~\eqref{LLE_system} about $\ubu$ equals $\mathcal{L}(\underline{\ub}) - \eps$, where $L(\underline{\ub}),\mathcal{L}(\underline{\ub}) \colon H^2(\R) \subset L^2(\R) \to L^2(\R)$ are the closed and densely defined operators given by
\begin{align*}
	L(\underline{\ub})= - \partial_x^2 +\zeta - 
	\begin{pmatrix}
		3 \underline{\ub}_1^2 + \underline{\ub}_2 ^2 & 2 \underline{\ub}_1 \underline{\ub}_2 \\
		2 \underline{\ub}_1 \underline{\ub}_2 &  \underline{\ub}_1^2 + 3\underline{\ub}_2 ^2
	\end{pmatrix}, \qquad \mathcal{L}(\underline{\ub}) = JL(\underline{\ub}).
\end{align*}

\subsection{Primary 1-pulse solutions}

Stationary $1$-pulse solutions of~\eqref{LLE_system} were constructed in~\cite{Gaertner2020Soliton,Gasmi2022Lugiato,Bengel2024Stability} by bifurcating from the rotated bright soliton solution~\eqref{bright_sol} of the NLS equation~\eqref{NLS}, where the rotation parameter $\theta \in \R$ has to satisfy the bifurcation equation $\pi f \cos\theta = 2 \sqrt{2\zeta}$. In the following we collect the relevant details from~\cite{Bengel2024Stability} on the existence and spectral stability of these $1$-pulse solutions, which will serve as building blocks for the upcoming construction of stationary multipulse solutions to~\eqref{LLE_system}. In order to formulate the result from~\cite{Bengel2024Stability}, we first state the definition of spectral stability for stationary pulse solutions of~\eqref{LLE_system}, as well as the definition of spectral instability.

\begin{Definition}
Let $\ub_\infty \in \R^2$ and let $\ubu \colon \R \to \R^2$ be a smooth stationary solution to~\eqref{LLE_system} such that $\ubu(x)$ converges to $\ub_\infty$ as $x \to \pm \infty$. The stationary pulse solution $\ubu$ to~\eqref{LLE_system} is \emph{spectrally stable} if there exists $\tau > 0$ such that
\begin{align*}
    \sigma(\El(\ubu)-\eps) \subset \{\lambda \in \C : \Re(\lambda) \leq -\tau\} \cup \{0\}
\end{align*}
and $0$ is an algebraically simple eigenvalue of $\El(\ubu) - \eps$.
\end{Definition}

\begin{Definition}
A smooth stationary bounded solution $\ubu \colon \R \to \R^2$ to~\eqref{LLE_system} is  \emph{spectrally unstable} if there exists $\lambda \in \sigma(\El(\ubu)-\eps) $ with $\Re(\lambda) > 0$. 
\end{Definition}

It has been proved in~\cite[Theorem~3]{Bengel2024Stability} that spectral stability yields nonlinear stability of pulse solutions to~\eqref{LLE_system} against $L^2$-localized perturbations, see also Theorem~\ref{thm:nonstab_Npulse}. We are now ready to summarize the existence and spectral stability results on $1$-pulse solutions from~\cite{Bengel2024Stability}. 

\begin{Theorem}[\!{\!\cite[Theorem 1,2]{Bengel2024Stability}}]\label{thm:1-pulse}
Assume that $\zeta,f>0$ and $\theta_0 \in \R$ obey $f^2 \pi^2> 8 \zeta$, $\pi f \cos\theta_0 = 2 \sqrt{2\zeta}$ and $\sin\theta_0 \neq 0$. Then, there exist $C_0, \eps_0 > 0$ such that for all $\eps \in (0,\eps_0)$ there exist an asymptotic state $\ub_{\infty,\eps} \in \R^2$ and an even smooth solution $\ubu_\eps \colon \R \to \R^2$ of \eqref{LLE_system_existence} satisfying
\begin{align} \label{estimates_1_pulse}
    \left\|\ubu_\eps - \boldsymbol{\phi}_{\theta_0}\right\|_{L^\infty}, |\ub_{\infty,\eps}| \leq C_0\eps, \qquad \ubu_\eps - \ub_{\infty,\eps} \in H^2(\R),
\end{align}
where we denote $\boldsymbol{\phi}_{\theta} = \phi_0 (\cos \theta,\sin \theta)^\top$ and $\phi_0$ is the bright soliton given by~\eqref{bright_sol}. The spectral stability of the solution $\ubu_\eps$ depends on the rotational angle $\theta_0$ as follows.
\begin{itemize}
    \item[(i)] If $\sin\theta_0>0$, then $\ubu_\eps$ is spectrally stable as a stationary pulse solution to~\eqref{LLE_system}.
    \item[(ii)] If $\sin\theta_0<0$, then $\ubu_\eps$ is spectrally unstable as a stationary solution to~\eqref{LLE_system}.
\end{itemize}
In both cases $\ker(\El(\ubu_\eps)-\eps)$ is spanned by $\ubu_\eps'$.
\end{Theorem}

\begin{figure}[t]
    \centering
    \includegraphics[width=0.3\textwidth]{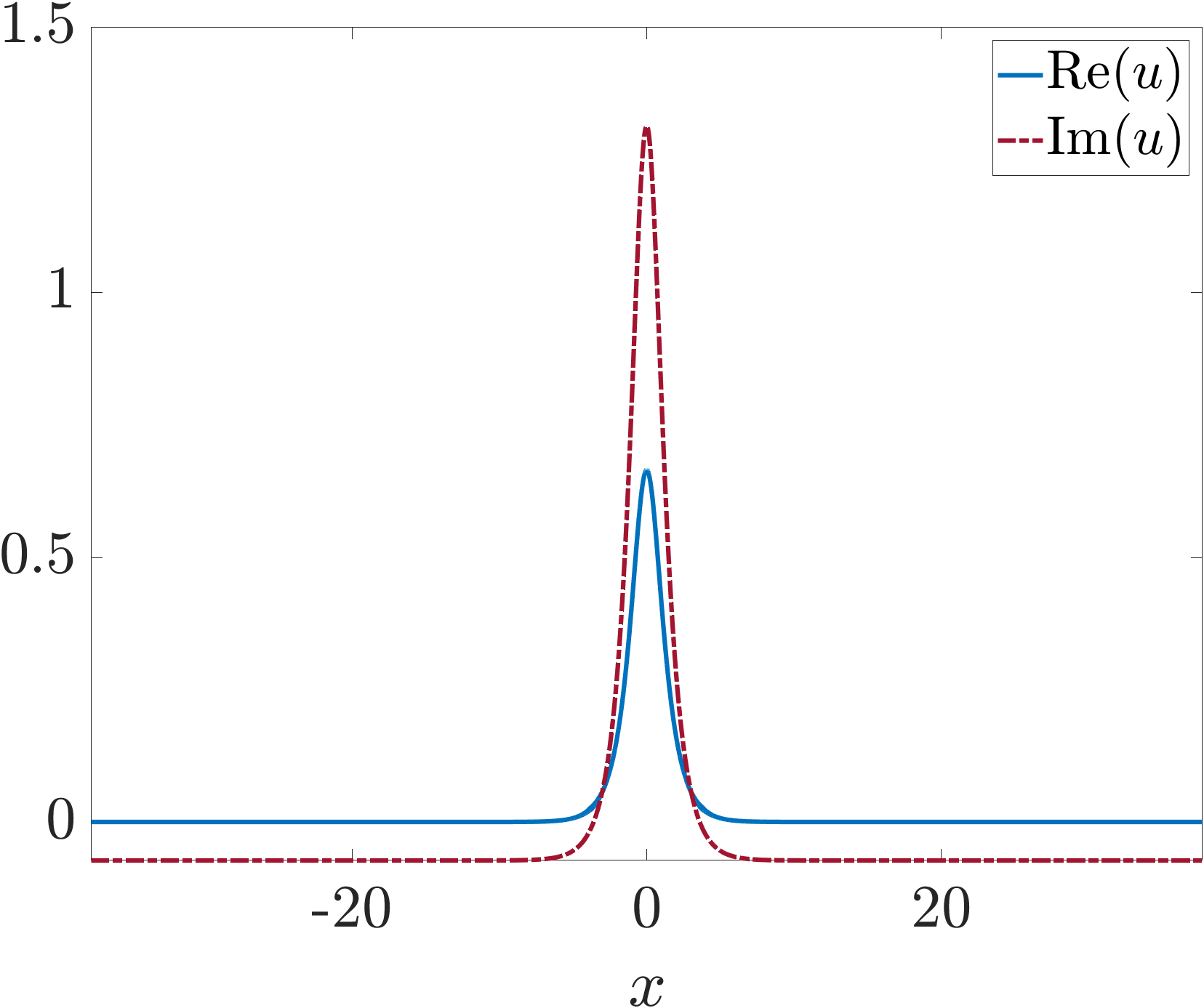}\hspace{1.5em}
    \includegraphics[width=0.3\textwidth]{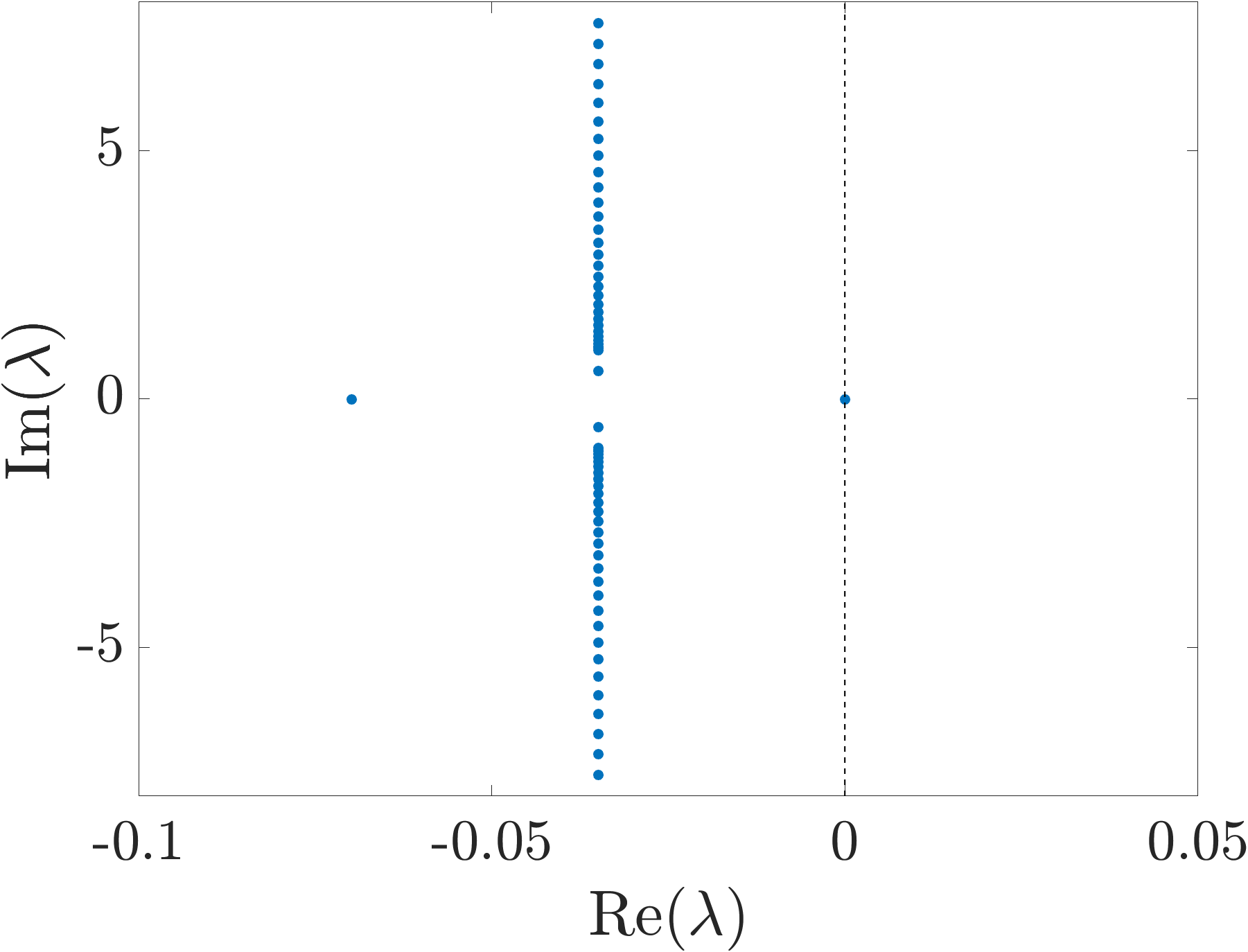}\hspace{1.5em}
    \includegraphics[width=0.3\textwidth]{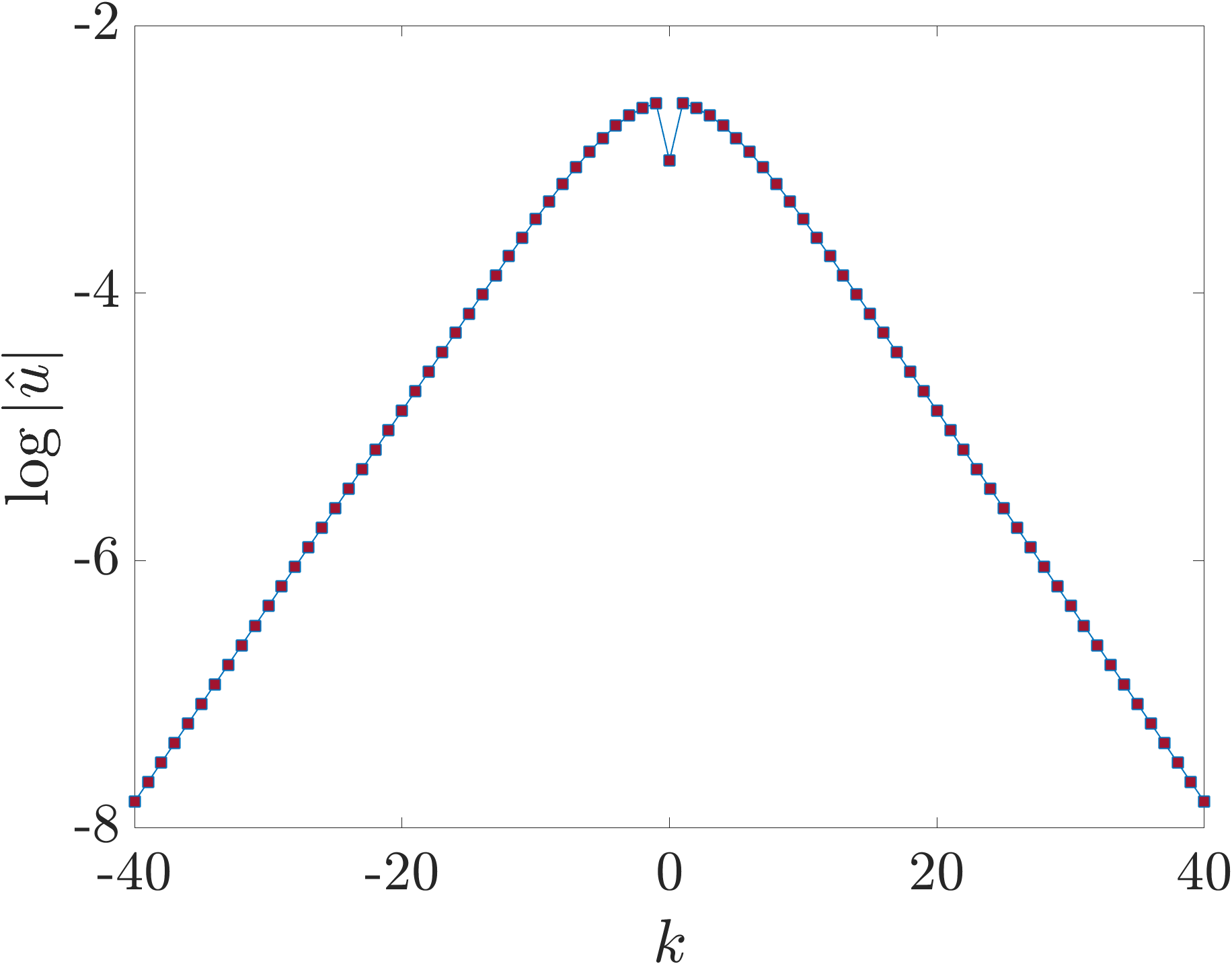}
    \\\vspace*{1em}
    \includegraphics[width=0.3\textwidth]{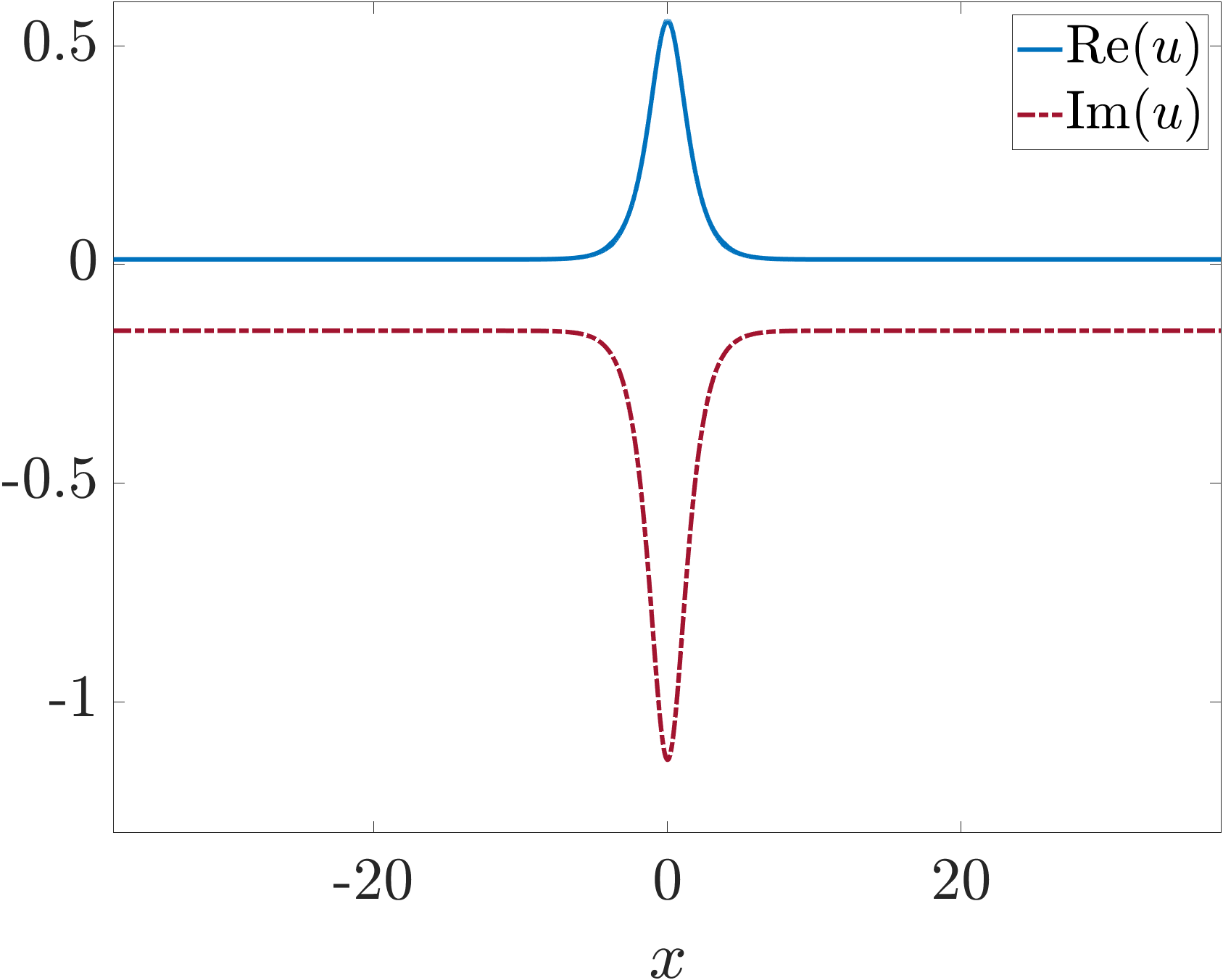}\hspace{1.5em}
    \includegraphics[width=0.3\textwidth]{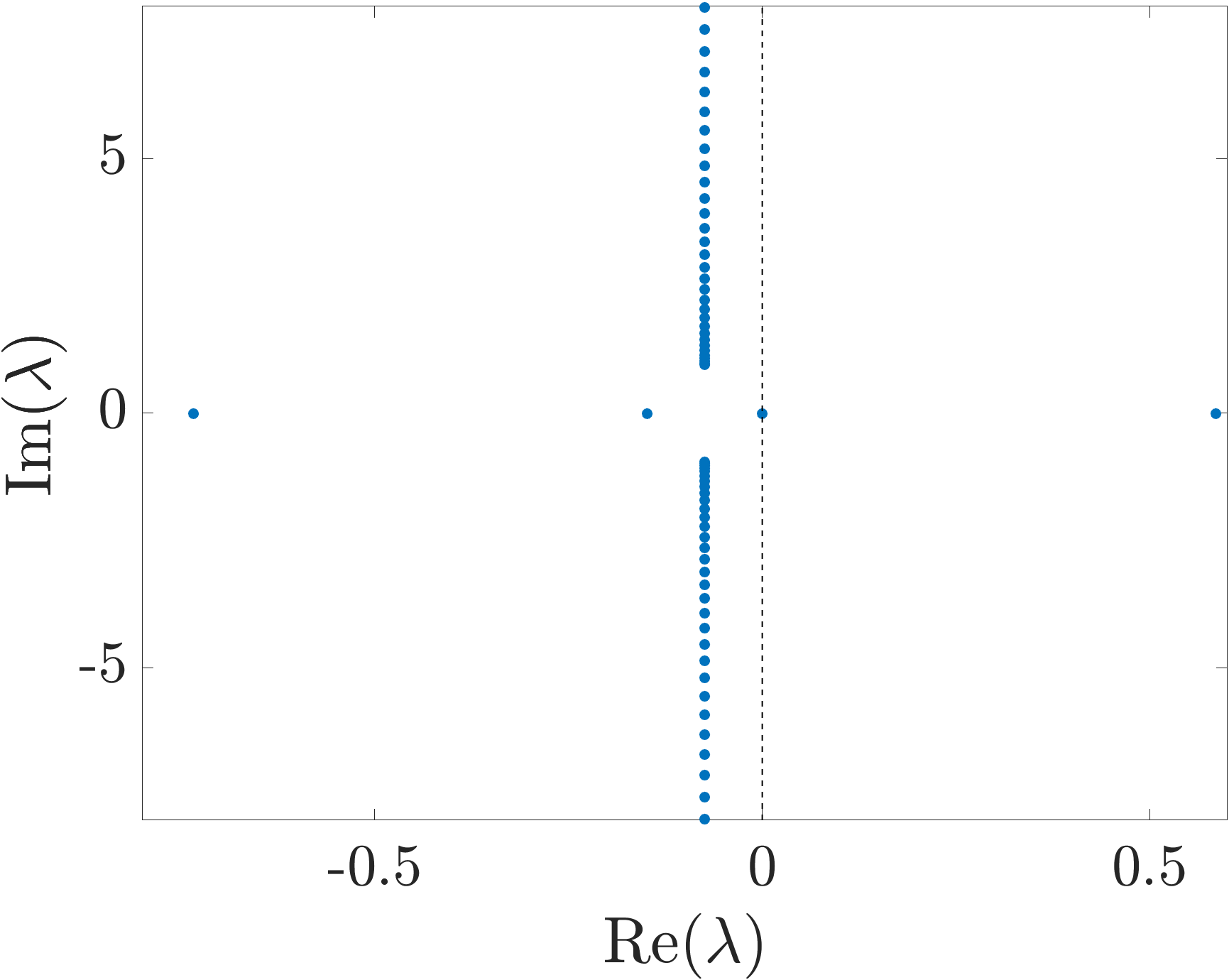}\hspace{1.5em}
     \includegraphics[width=0.3\textwidth]{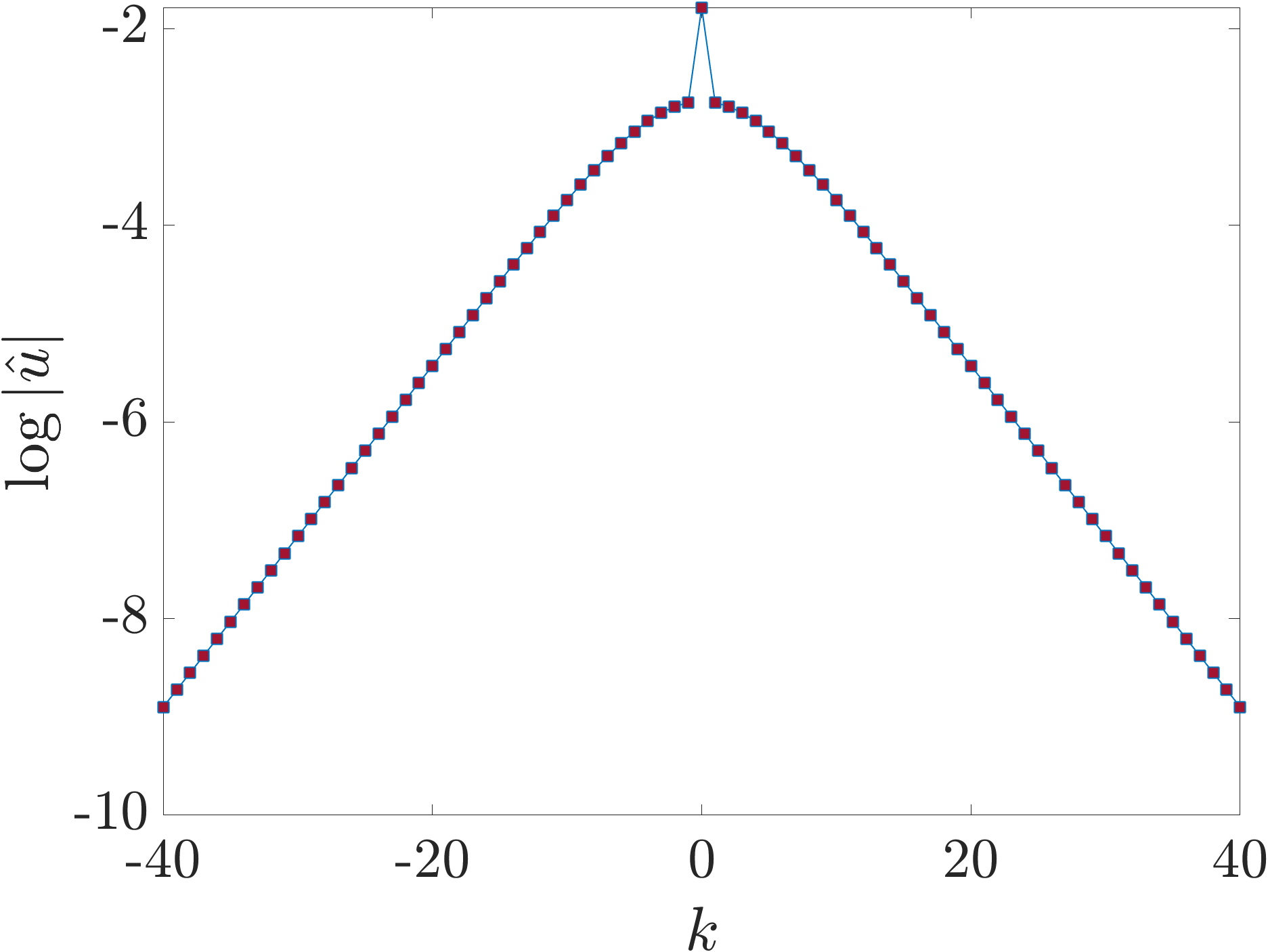}
    \caption{Periodic approximations of the primary 1-pulse solutions established in Theorem~\ref{thm:1-pulse} (see also Theorem~\ref{thm:ex_periodic}) with parameters $\zeta =1, f=2$. The solutions were computed with the MATLAB package \texttt{pde2path} \cite{Uecker2014pde2path}. The top row shows a stable 1-pulse bifurcating with $\sin\theta_0>0$, its spectrum against co-periodic perturbations, and the corresponding frequency comb obtained by plotting $\log|\hat{u}(k)|$ against the Fourier frequency variable $k$. The bottom row depicts an unstable 1-pulse bifurcating with $\sin\theta_0<0$, the associated spectrum, and its frequency comb.
    }
\end{figure}

We prove that the $1$-pulse solutions, established in Theorem~\ref{thm:1-pulse}, correspond to nondegenerate homoclinics connecting to a saddle-focus equilibrium in the dynamical system~\eqref{LLE_DS}. The definition of nondegeneracy, cf.~\cite{Vanderbauwhede1992Homoclinic}, reads as follows. 

\begin{Definition}
A homoclinic solution $\un{U}$ of~\eqref{LLE_DS} is called \emph{nondegenerate} if all bounded solutions $U \colon \R \to \R^4$ to the variational problem
\begin{align} \label{e:variational}
    U' = \partial_U F(\un{U}(x);\eps) U
\end{align}
are given by scalar multiples of $\un{U}'$.
\end{Definition}

We exploit the spectral properties of the linearization of~\eqref{LLE_system} about the primary $1$-pulse to show nondegeneracy of the corresponding homoclinic in~\eqref{LLE_DS}. In particular, the nondegeneracy follows from the fact that $0$ is a geometrically simple eigenvalue of the linearization. 

\begin{Lemma}\label{lem:nondegeneracy}
Let $\ub_\infty \in \R^2$ and let $\ubu \colon \R \to \R^2$ be a smooth solution of~\eqref{LLE_system_existence} such that $\ubu(x) \to \ub_\infty$ and $\ubu'(x) \to 0$ as $x \to \pm \infty$. Assume that $0$ is a geometrically simple eigenvalue of $\El(\ubu)-\eps$ with associated eigenfunction $\ubu'$. Then, the corresponding homoclinic solution $\un{U} \colon \R \to \R^4$ to~\eqref{LLE_DS} given by $\un{U} = \smash{(\ubu,\ubu')^\top}$  is nondegenerate if its asymptotic state $U_\infty = \smash{(\ub_\infty,0)^\top}$ is hyperbolic.
\end{Lemma}
\begin{proof}
First, we note that, since the variational equation~\eqref{e:variational} has smooth coefficients, all of its solutions are smooth. Moreover, if $U = (U_1,U_2,U_3,U_4)^\top$ is an $L^2$-localized solution of~\eqref{e:variational}, then it follows, by expressing derivatives through the equation, that $U'$ and $U''$ are also $L^2$-localized. Hence, we infer that $\ub =(U_1,U_2)^\top \in H^2(\R)$ must lie in $\ker(\El(\ubu) - \eps) = \spann\{\ubu'\}$. Consequently, $U$ is a scalar multiple of $\un{U}'$, implying that $\un{U}$ is nondegenerate. Thus, in order to prove the result, it suffices to show that all bounded solutions of~\eqref{e:variational} are $L^2$-localized. This will be achieved with the aid of exponential dichotomies. We first look at the limiting system
\begin{align} \label{e:variationallimit}
    U' = \partial_U F(U_\infty;\eps) U.
\end{align}
Since the matrix $\partial_U F(U_\infty;\eps)$ is hyperbolic by assumption, system~\eqref{e:variationallimit} admits an exponential dichotomy on $\R$. Combining the latter with the fact that the coefficient matrix $\partial_U F(\un{U}(x);\eps)$ of~\eqref{e:variational} converges to $\partial_U F(U_\infty;\eps)$ as $x \to \pm \infty$, we infer that system~\eqref{e:variational} possesses exponential dichotomies on both half-lines $(-\infty,0]$ and $[0,\infty)$ by~\cite[Lemma~3.4]{Palmer1984Exponential}. Therefore, every bounded solution of~\eqref{e:variational} is exponentially localized and, thus, lies in $L^2(\R)$.
\end{proof}

With the aid of Lemma~\ref{lem:nondegeneracy}, we establish that the primary $1$-pulse solutions in Theorem~\ref{thm:1-pulse} correspond to nondegenerate homoclinics in~\eqref{LLE_DS} connecting to a saddle-focus equilibrium.

\begin{Proposition}\label{prop:saddle_focus}
Let $\ubu_\eps$ be a stationary $1$-pulse solution, established in Theorem~\ref{thm:1-pulse}. Then, provided $\eps > 0$ is sufficiently small,
\begin{align*} \underline{U}_\eps = (\ubu_{1,\eps},\ubu_{2,\eps},\ubu_{1,\eps}',\ubu_{2,\eps}')^\top\end{align*} 
is a nondegenerate homoclinic solution to~\eqref{LLE_DS}, whose asymptotic state $U_{\infty,\eps} = \smash{\displaystyle \lim_{x \to \pm\infty} \un{U}_\eps(x)}$ is a saddle-focus equilibrium, i.e., $\partial_U F(U_{\infty,\eps};\eps)$ has the four eigenvalues $\pm \alpha \pm \beta \iu$ with $\alpha,\beta> 0$. 
\end{Proposition}
\begin{proof}
It holds $F(0;0) = 0$ and $\text{det}(\partial_U F(0;0)) = \zeta^2 > 0$. So, with the aid of the implicit function theorem, we find a unique equilibrium $U_{\infty,\eps} \in \R^4$ of~\eqref{LLE_DS} converging to $0 \in \R^4$ as $\eps \to 0$.
On the other hand, since $\ubu_\eps - \ub_{\infty,\eps}$ lies in $H^2(\R)$ by Theorem~\ref{thm:1-pulse} and functions in $H^1(\R)$, being square integrable and uniformly continuous, converge to $0$ as $x \to \pm \infty$, it follows that $\un{U}_\eps$ is a homoclinic solution of~\eqref{LLE_DS} connecting to the equilibrium $(\ub_{\infty,\eps},0)^\top \in \R^4$. Since the equilibrium $(\ub_{\infty,\eps},0)^\top$ converges to $0$ as $\eps \to 0$ by Theorem~\ref{thm:1-pulse}, it must hold $U_{\infty,\eps} = (\ub_{\infty,\eps},0)^\top$. One readily observes that the matrix $\partial_UF(U_{\infty,\eps};\eps) \in \R^{4 \times 4}$ possesses four eigenvalues $\pm \nu_{\pm,\eps} \in \C$ satisfying
\begin{align*}
	\nu_{\pm,\eps}^2 = \zeta - 2 |\ub_{\infty,\eps}|^2 \pm \iu \sqrt{\eps^2 - |\ub_{\infty,\eps}|^4}.
 \end{align*}
Since we have $|\ub_{\infty,\eps}| = O(\eps)$ by Theorem~\ref{thm:1-pulse}, these four eigenvalues are, provided $\eps > 0$ is sufficiently small, of the form $\pm\alpha \pm \iu \beta$ with $\alpha,\beta >0$, implying that $U_{\infty,\eps}$ is a saddle focus. Finally, using that $U_{\infty,\eps}$ is hyperbolic and $0$ is a simple eigenvalue of $\El(\ubu_\eps) - \eps$ by Theorem~\ref{thm:1-pulse}, we infer that the homoclinic $\un{U}_\eps$ is nondegenerate. 
\end{proof}

\subsection{\texorpdfstring{$N$-pulse}{N-pulse} solutions} \label{ssec:ex_N-pulse}

In this subsection we establish stationary multipulse solutions to~\eqref{LLE_system} composed of any number $N \in \N$ of well-separated primary $1$-pulses, which were obtained in Theorem~\ref{thm:1-pulse}. We allow for superpositions of $1$-pulses bifurcating from bright solitons with different phase rotations. The constructed $N$-pulses are even and the individual distances between the first $\lfloor \frac{N}{2} \rfloor + 1$ pulses can be chosen arbitrarily large, independently of each other. The linearization of~\eqref{LLE_system} about the $N$-pulses possesses $N$ eigenvalues, which converge to $0$ as the distances between the individual pulses tend to $\infty$. These $N$ small eigenvalues are associated with the translational eigenvalues of the $N$ individual 1-pulses constituting the multipulse. We employ~\cite[Theorem~3.6]{Sandstede1999Instability} and~\cite[Theorem~1]{Sandstede1998Stability}, see also~\cite[Section~3]{Sandstede1993Verzweigungstheorie}, to establish for any number $\ell \in \{0,\ldots,N-1\}$ the existence of an $N$-pulse solution such that there are precisely $\ell$ stable eigenvalues and $N-1-\ell$ unstable eigenvalues amongst the $N$ small eigenvalues. The $N$-th eigenvalue resides at the origin and is algebraically simple, implying that the associated $N$-homoclinic in the dynamical system~\eqref{LLE_DS} is nondegenerate, cf.~Lemma~\ref{lem:nondegeneracy}. In summary, given a sequence of $\lceil \frac{N}{2} \rceil$ primary $1$-pulses, an application of~\cite[Theorem~1]{Sandstede1993Verzweigungstheorie} and~\cite[Theorem~3.6]{Sandstede1999Instability} yield an $(n+1)$-parameter family of associated even $N$-pulses, where the first $n = \lfloor \frac{N}{2} \rfloor$ parameters regulate the distances between the first $n+1$ pulses and the last parameter corresponds to the number of stable small eigenvalues.

\begin{Theorem}\label{thm:ex_N-pulse}
Let $N \in \N$ and $\ell \in \{0,\ldots,N-1\}$. Set $n = \lfloor \frac{N}{2} \rfloor$ and $\alpha_0 = N\! \mod 2 \in \{0,1\}$. Let $\ubu_1,\ldots,\ubu_{n+\alpha_0}$ be a sequence of stationary $1$-pulse solutions, established in Theorem~\ref{thm:1-pulse} and converging to the asymptotic end state $\ub_\infty \in \R^2$ as $x \to \pm \infty$. Then, there exist constants $\delta_0, C_0 > 0$ such that for each $\mathbf{k} = (k_1,\ldots,k_n) \in \N^n$ there exist distances $\smash{T_{1,\ell}^{k_1},\ldots,T_{n,\ell}^{k_n} > 0}$ such that~\eqref{LLE_system_existence} admits an even smooth solution $\ubu_{\mathbf{k},\ell} \colon \R \to \R^2$ satisfying 
\begin{align} \label{boundmultipulse}
\begin{split}
&\left|\ubu_{\mathbf{k},\ell}(x) - \ub_\infty - \alpha_0 \left(\ubu_{n+\alpha_0}(x) - \ub_\infty\right) - \sum_{i = 1}^n \left(\ubu_i\left(x - T_{1,\ell}^{k_1} - \ldots - T_{i,\ell}^{k_i}\right)\right. \right.\\ 
&\qquad \qquad \qquad \qquad \qquad \quad \qquad \left.  \phantom{ \sum_{i = 1}^n }\left. +\, \ubu_i\left(x + T_{1,\ell}^{k_1} + \ldots + T_{i,\ell}^{k_i}\right) - 2\ub_\infty\right) \right| \leq \frac{C_0}{\min\{k_1,\ldots,k_n\}}, 
\end{split}
\end{align}
for $x \in \R$. The spectrum of $\El(\ubu_{\mathbf{k},\ell}) - \eps$ in the ball $B_{\delta_0}(0)$ consists of $N$ eigenvalues of which $\ell$ have negative real part and $N-1-\ell$ have positive real part (all counted with algebraic multiplicities). Moreover, $0$ is an algebraically simple eigenvalue. In addition, $\un{U}_{\mathbf{k},\ell} = (\ubu_{\mathbf{k},\ell},\ubu_{\mathbf{k},\ell}')^\top$ is a  nondegenerate homoclinic solution to~\eqref{LLE_DS}, whose asymptotic end state $U_\infty = \smash{\displaystyle \lim_{x \to \pm \infty} \un{U}_{\mathbf{k},\ell}(x)}$ is a saddle-focus equilibrium. Finally, the sequence $\smash{\{T_{i,\ell}^{k}\}_k}$ of pulse distances is monotonically increasing with $\smash{T_{i,\ell}^k} \to \infty$ as $k \to \infty$ for $i = 1,\ldots,n$. 
\end{Theorem}

\begin{proof}
The proof follows from the homoclinic bifurcation results,~\cite[Theorem~1]{Sandstede1998Stability} and~\cite[Theorem~3.6]{Sandstede1999Instability}, after verifying that all necessary hypotheses are satisfied. Specifically, Theorem~\ref{thm:1-pulse} and Proposition~\ref{prop:saddle_focus} yield that $0$ is an algebraically simple eigenvalue of the linearization $\El(\ubu_j) - \eps$ and $U_j = \smash{(\ubu_j,\ubu_j')^\top}$ is a symmetric homoclinic solution to~\eqref{LLE_DS} connecting to the saddle-focus equilibrium $U_\infty = \smash{(\ub_\infty,0)^\top}$ for $j = 1,\ldots,n+\alpha_0$. Furthermore, the asymptotic behavior of $U_j$, $U_j'$, and all bounded solutions to the adjoint problem $\Psi' = - \partial_U F(U_j;\eps)^* \Psi$ is fully described by the eigenvalues of the matrix $\partial_U F(U_\infty;\eps) \in \R^{4 \times 4}$, which are of the form $\pm \alpha \pm \iu \beta$ with $\alpha,\beta > 0$ as $U_\infty$ is a saddle focus. We deduce that the (un)stable manifold of the equilibrium $U_\infty$ in system~\eqref{LLE_DS} do not admit any strong (un)stable submanifolds by dimension counting. Hence, the stable and unstable manifolds along the homoclinics $U_j$ are not in an orbit-flip configuration, cf.~\cite[Section 2.1]{Homburg2010Homoclinic}. By an analogous dimension counting argument we deduce that the (un)stable manifolds along the homoclinics $U_j$ are not in an inclination-flip configuration. We thus conclude that~\cite[Theorem~1]{Sandstede1998Stability} and~\cite[Theorem~3.6]{Sandstede1999Instability} apply, which establish the result, where we point out that the nondegeneracy of $\un{U}_{\mathbf{k},\ell}$ follows from Lemma~\ref{lem:nondegeneracy}.
\end{proof}

We emphasize that taking $\ell \neq N-1$ in Theorem~\ref{thm:ex_N-pulse} results in spectrally unstable $N$-pulses, even if all primary 1-pulses $\ubu_1,\ldots,\ubu_{n+\alpha_0}$ are spectrally stable. 

\begin{Remark}\label{rem:long_term_instability}
In~\cite{Sandstede1998Stability} Lin's method is employed to determine leading-order expressions of the $N$ small eigenvalues bifurcating from $0$. The leading-order expressions in~\cite[Theorem~2]{Sandstede1998Stability} imply that there exist constants $C,\eta > 0$ such that that any of the $N$ eigenvalues $\lambda \in B_{\delta_0}(0)$ in Theorem~\ref{thm:ex_N-pulse} obeys the estimate $|\lambda| \leq C\exp(-\eta \min\{T_{1,\ell}^{k_1},\ldots,T_{n,\ell}^{k_n}\})$. Particularly, any potential instabilities arising from the eigenvalues near zero can be interpreted as \emph{long-time instabilities}, i.e.~instabilities that are only observed on time scales which are exponentially long in terms of the pulse distances.
\end{Remark}

\begin{figure}[t]
    \centering
    \includegraphics[width=0.3\textwidth]{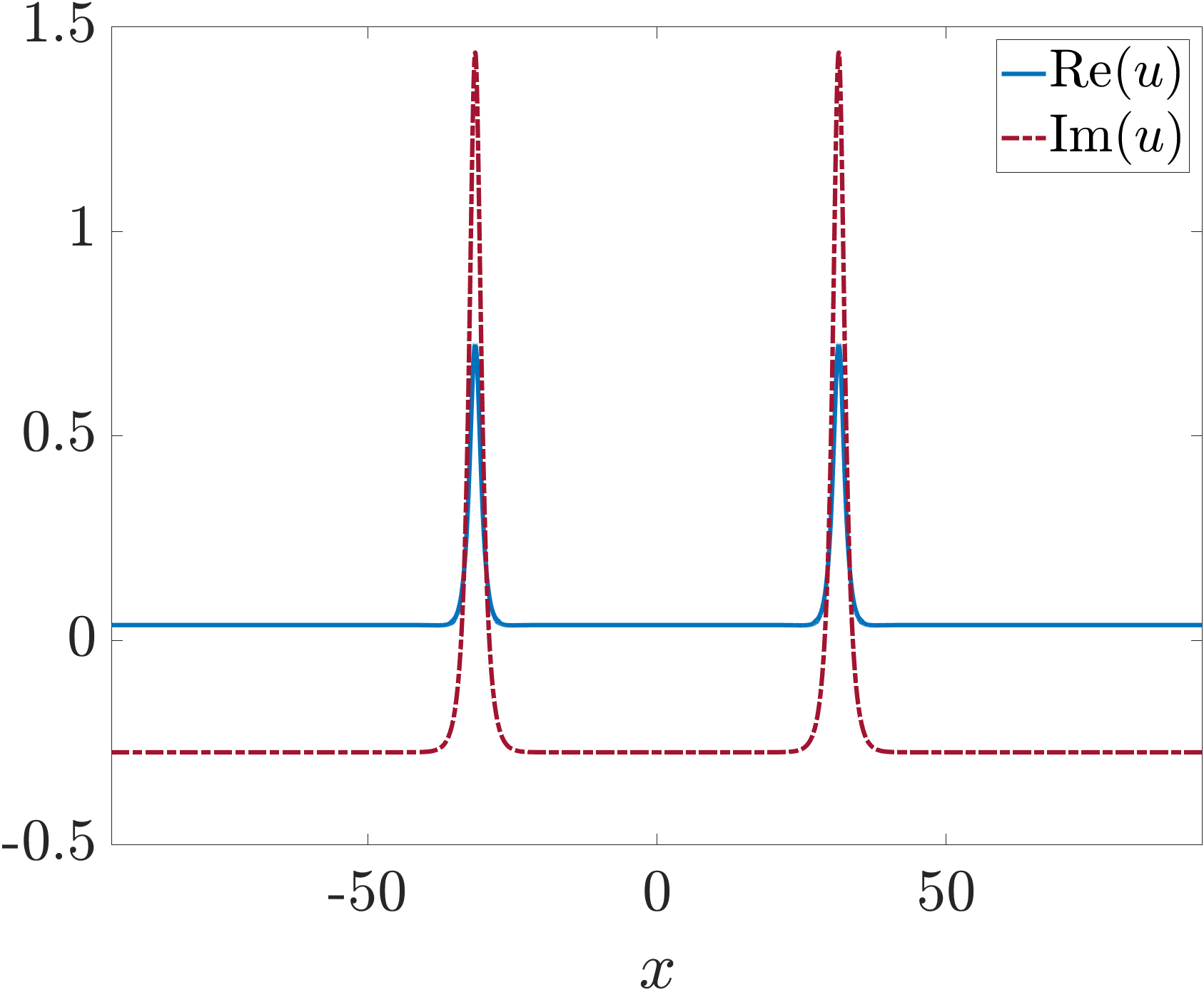}\hspace{1.5em}
    \includegraphics[width=0.3\textwidth]{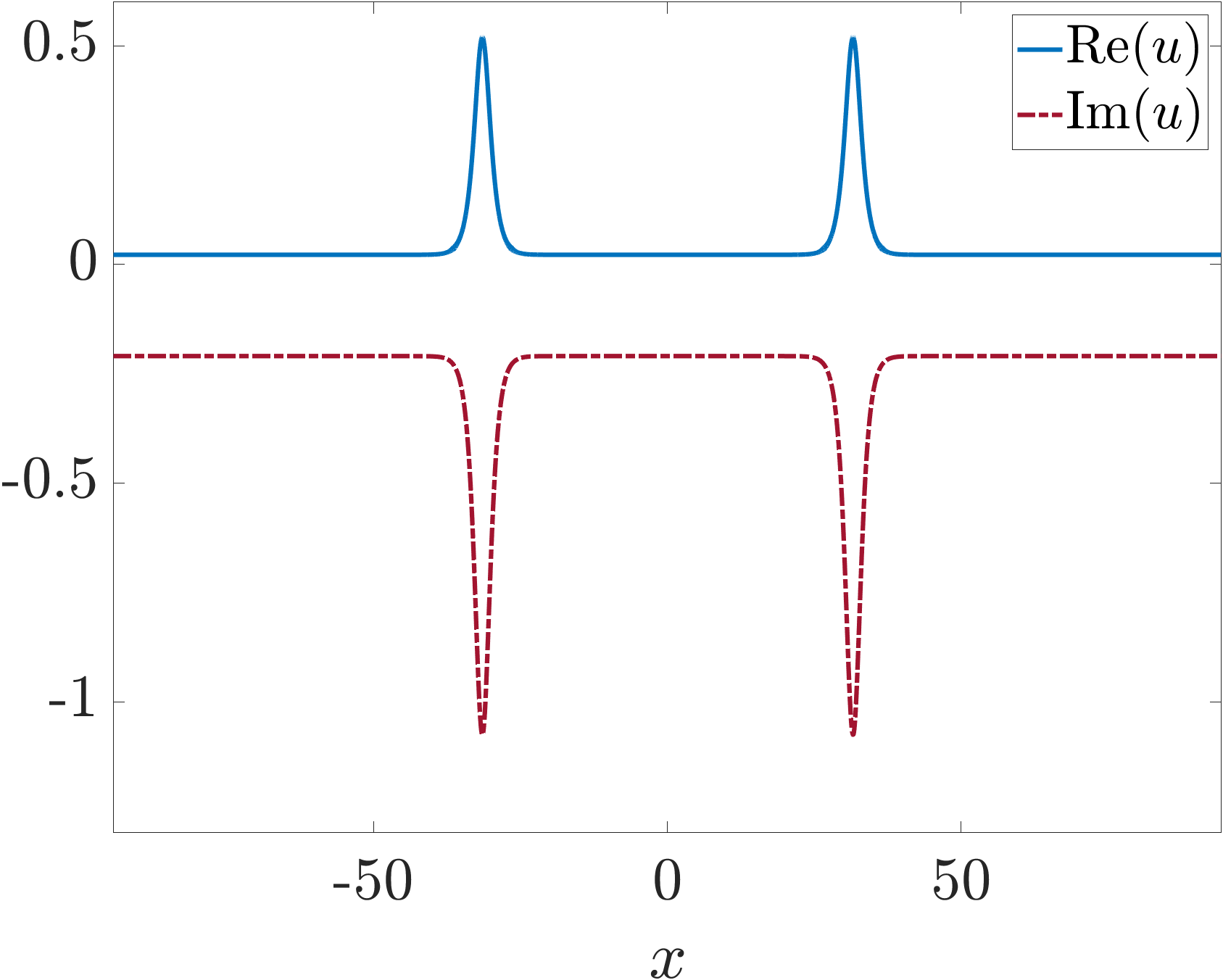}\hspace{1.5em}
    \includegraphics[width=0.3\textwidth]{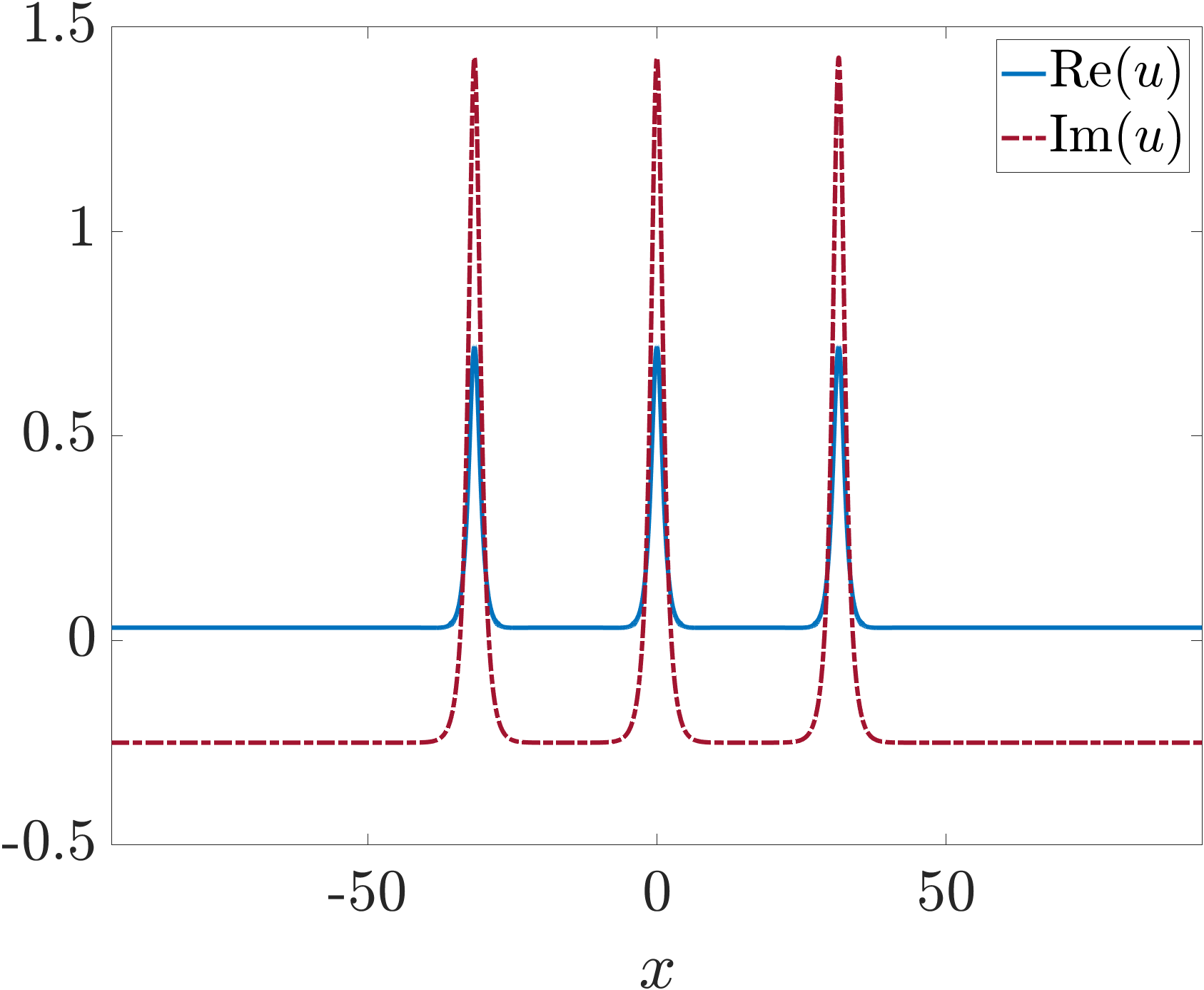}\\\vspace*{1em}
    \includegraphics[width=0.3\textwidth]{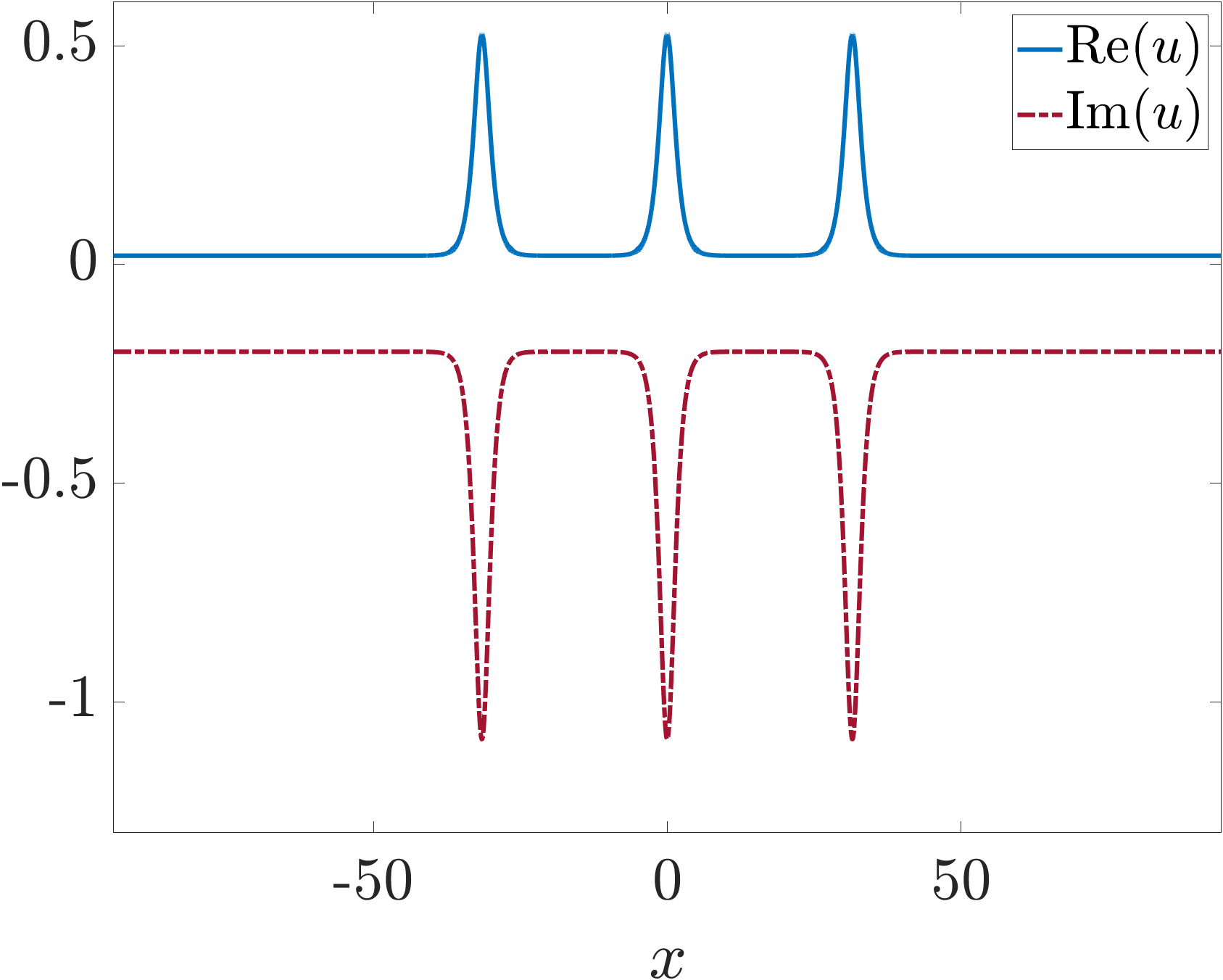}\hspace{1.5em}
    \includegraphics[width=0.3\textwidth]{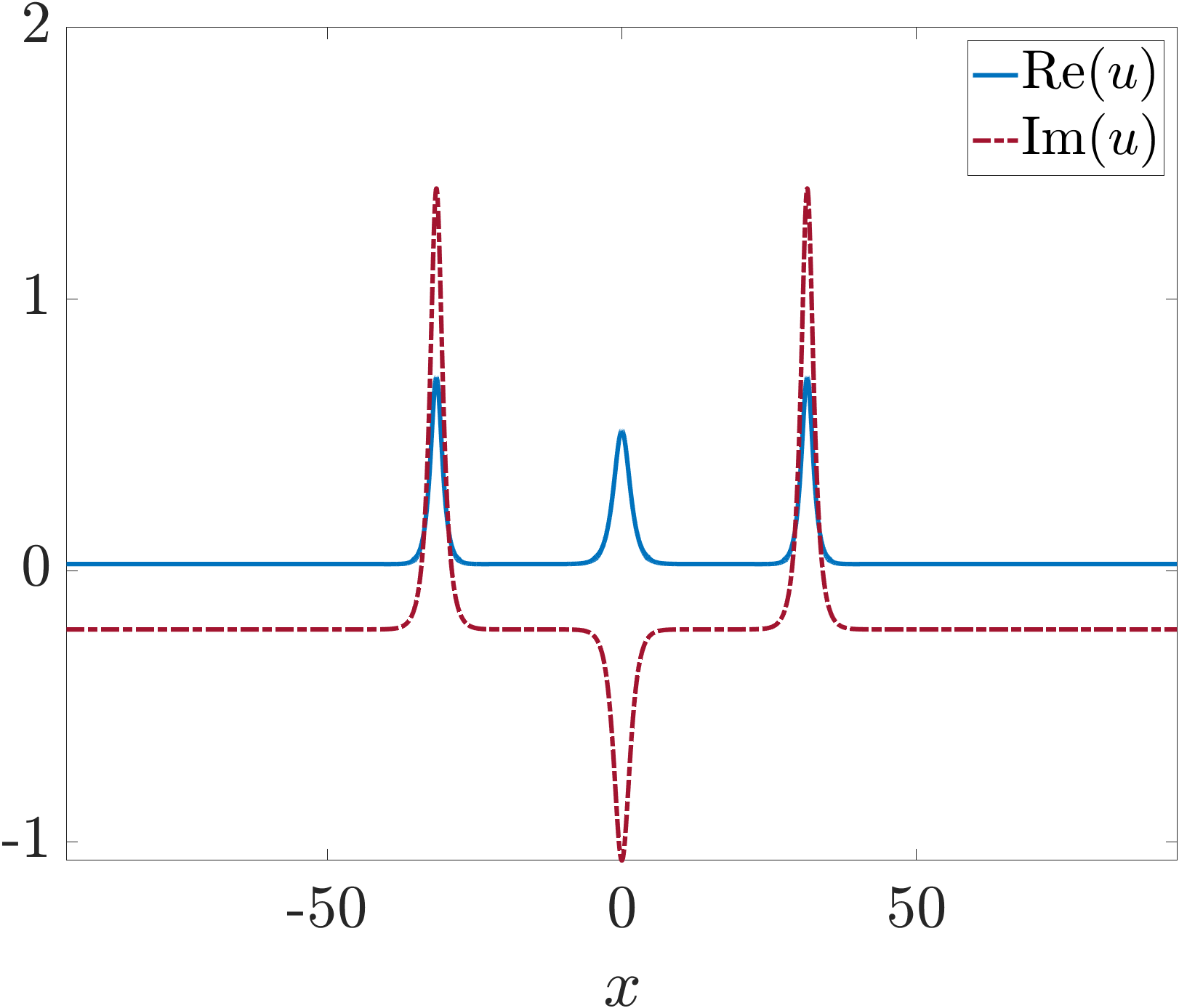}\hspace{1.5em}
    \includegraphics[width=0.3\textwidth]{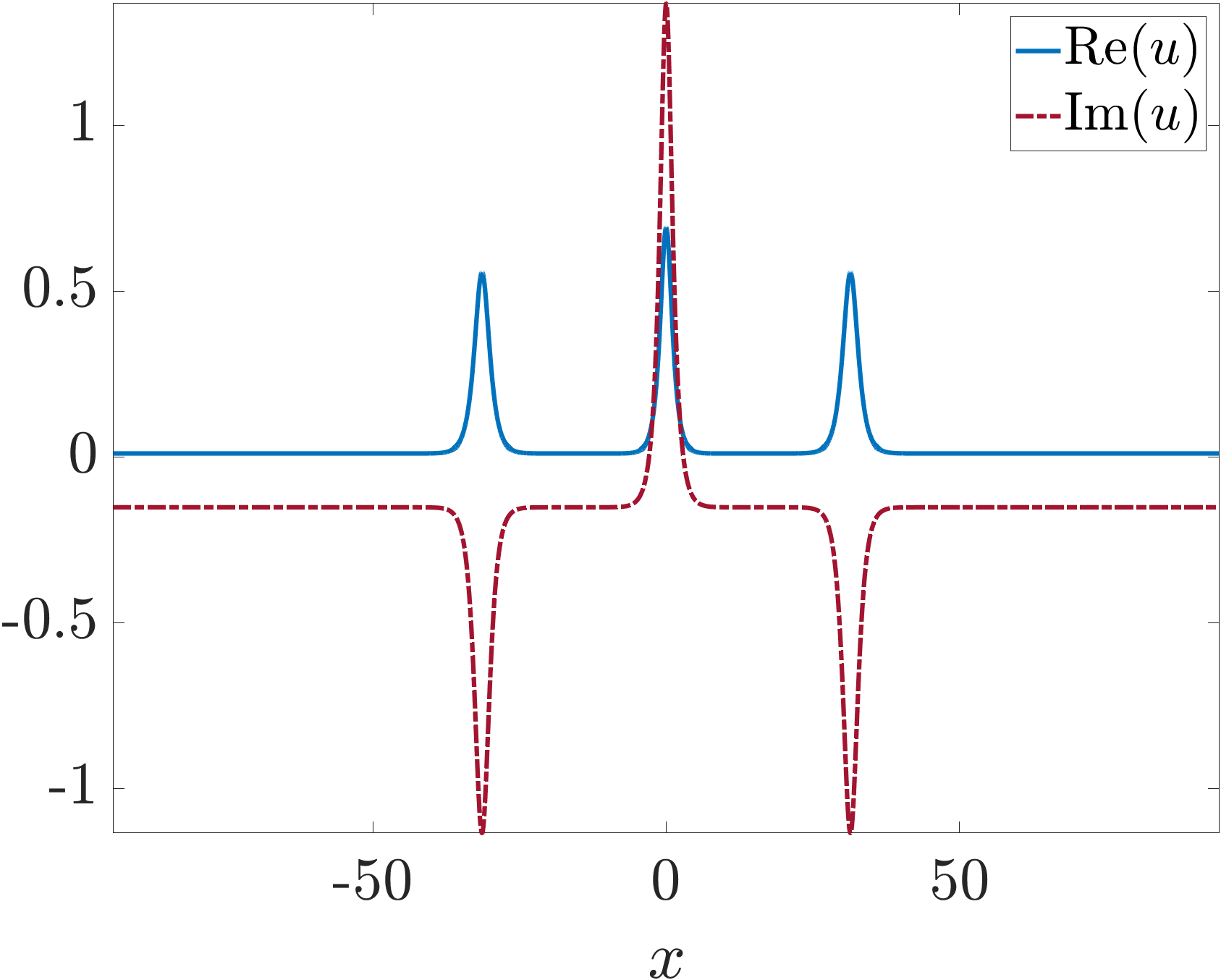}
    \caption{The figure shows periodic approximations of 2- and 3-pulses, which were established in Theorem~\ref{thm:ex_N-pulse}, see also Theorem~\ref{thm:ex_periodic}, computed with \texttt{pde2path} \cite{Uecker2014pde2path}. The parameters are $\zeta =1$ and $f=2$. Continuation was performed in the small parameter $\eps$ starting from superpositions of bright NLS-solitons.}
\end{figure}

\subsection{Periodic \texorpdfstring{$N$-pulse}{N-pulse} solutions} \label{ssec:ex_periodic_pulse}

We obtain periodic multipulse solutions to~\eqref{LLE_system_existence} by employing the dynamical systems formulation~\eqref{LLE_DS}. Specifically, we use that any nondegenerate symmetric homoclinic in a reversible dynamical system connecting to a saddle-focus equilibrium is accompanied by a $1$-parameter family of symmetric periodic orbits parameterized by the period $T$. The $T$-periodic orbits converge uniformly to the homoclinic on a single periodicity interval as $T \to \infty$. This result, which was obtained in~\cite[Theorem~5]{Vanderbauwhede1992Homoclinic}, follows from an application of Lin's method.

\begin{Theorem}\label{thm:ex_periodic}
Let $\un{U} \colon \R \to \R^4$ be a nondegenerate symmetric solution of~\eqref{LLE_DS} homoclinic to a saddle-focus equilibrium $U_\infty \in \R^4$, as established in Theorem~\ref{thm:1-pulse} or Theorem~\ref{thm:ex_N-pulse}. Then, there exists $T_0 > 0$ such that for every $T \geq T_0$ there exists a smooth $T$-periodic symmetric solution $\un{U}_T \colon \R \to \R^4$ to~\eqref{LLE_DS} satisfying 
\begin{align}\label{eq:limit_periodic_to_homoclinic}
    \lim_{T \to \infty} \sup_{x \in \big[-\tfrac{T}{2},\tfrac{T}{2}\big]} |\un{U}_T(x) - \underline{U}(x)| = 0.
\end{align}
\end{Theorem}
\begin{proof}
The proof follows immediately from~\cite[Theorem 5]{Vanderbauwhede1992Homoclinic} upon noting that nondegenerate homoclinic solutions are elementary by~\cite[page~302]{Vanderbauwhede1992Homoclinic}.
\end{proof}

Combining Theorem~\ref{thm:ex_periodic} with Theorem~\ref{thm:ex_N-pulse} and Proposition~\ref{prop:saddle_focus} we readily establish the existence of periodic $N$-pulse solutions to~\eqref{LLE_system_existence} for any $N \in \mathbb{N}$.

\section{Stability analysis}\label{sec:stability}

In this section, we establish the spectral and nonlinear stability of the multiple and periodic pulse solutions to~\eqref{LLE_system}, which were obtained in Theorems~\ref{thm:ex_N-pulse} and~\ref{thm:ex_periodic}. As outlined in~\S\ref{sec:approach}, the spectral stability analysis of these pulse solutions hinges on a-priori bounds on the spectrum, a detailed assessment of the spectrum in a neighborhood of the origin relying on~\cite{Sandstede1999Instability,Sandstede2001Stability}, and Evans-function arguments from~\cite{BengeldeRijk} to control the spectrum in a compact set away from the origin. After having obtained spectral stability, nonlinear stability of the multiple and periodic pulse solutions follows by invoking results from~\cite{Bengel2024Stability} and~\cite{Haragus2023Nonlinear,Stanislavova2018Asymptotic}, respectively. 

\subsection{Spectral a-priori bounds}

Given a constant $\rho > 0$ and a smooth stationary solution $\ubu$ of~\eqref{LLE_system} obeying $\|\ubu\|_{L^\infty} \leq \rho$, we establish a-priori bounds on the spectrum of the linearization $\El(\ubu) - \eps$ of~\eqref{LLE_system} about $\ubu$. The bounds ensure that the spectrum of $\El(\ubu) - \eps$ in the closed right-half plane is confined to a compact set, whose boundary depends on $\rho$ and the detuning parameter $\zeta$ only. 

\begin{Lemma}\label{lem:a-priori_spectrum}
Fix $\rho > 0$ and $\zeta \in \R$. There exist constants $\eta_1,\eta_2 >0$ such that for all $\ubu \in L^\infty(\R)$ with $\|\ubu\|_{L^\infty} \leq \rho$ the set
\begin{align*}
    \Omega = \{\lambda \in \C : |\Re(\lambda)| \geq \eta_1\} \cup \{\lambda \in \C: |\Im(\lambda)| \geq \eta_2, \Re(\lambda) \neq 0\}
\end{align*}
belongs to the resolvent set of $\mathcal{L}(\ubu)$, i.e., we have $\sigma(\mathcal{L}(\ubu)) \cap \Omega = \emptyset$.
\end{Lemma}

\begin{proof}
The densely defined skew-adjoint operator $-J\partial_x^2 \colon H^2(\R) \subset L^2(\R) \to L^2(\R)$ generates a unitary group on the Hilbert space $L^2(\R)$ by Stone's theorem, see~\cite[Theorem~II.3.24]{Engel2000One}. In particular,~\cite[Theorem~I.1.10]{Engel2000One} yields the resolvent bound
\begin{align} \label{resolventprincipal}
    \left\|\left(-J\partial_x^2 - \lambda\right)^{-1}\right\|_{L^2 \to L^2} \leq \frac{1}{|\mathrm{Re}(\lambda)|}
\end{align}
for $\lambda \in \C$ with $|\mathrm{Re}(\lambda)| > 0$. The residual $\mathcal{L}(\ubu) + J\partial_x^2$ can be bounded as
\begin{align*}
    \left\|\left(\mathcal{L}(\ubu) + J\partial_x^2\right) \ub \right\|_{L^2} &\leq  \left(|\zeta| + 4 \|\ubu\|_{L^\infty}^2\right) \|\ub\|_{L^2} \leq C_1 \|\ub\|_{L^2} ,
\end{align*}
for all $\ub \in H^2(\R)$, where we set $C_1 = |\zeta| + 4 \rho^2$. By estimate~\eqref{resolventprincipal} there exists a constant $\eta_1 > 0$, depending on $\rho$ and $\zeta$ only, such that for $\lambda \in \C$ with $|\mathrm{Re}(\lambda)| \geq \eta_1$ we have
\begin{align*}
    C_1 \left\|(-J\partial_x^2 - \lambda)^{-1}\right\|_{L^2 \to L^2} < 1.
\end{align*} 
Therefore, $\mathcal{L}(\ubu) - \lambda = - J\partial_x^2 - \lambda + \big(\mathcal{L}(\ubu) + J\partial_x^2\big)$ is by~\cite[Theorem~IV.1.16]{Kato1995Perturbation} bounded invertible for each $\lambda \in \C$ with $|\mathrm{Re}(\lambda)| \geq \eta_1$. This yields
\begin{align*}
    \{\lambda \in \C : |\Re(\lambda)| \geq \eta_1\} \cap \sigma(\mathcal{L}(\ubu)) = \emptyset.
\end{align*}

Next, let $\lambda \in \C$ and $\mathbf{w} \in L^2(\R)$. Consider the resolvent problem
\begin{align}\label{resolvent_problem}
    (\mathcal{L}(\ubu) -\lambda) \ub = \mathbf{w}.
\end{align}
We show that there exists $\eta_2 >0$, depending on $\rho$ and $\zeta$ only, such that~\eqref{resolvent_problem} has a unique solution $\ub \in H^2(\R)$ provided $|\Im(\lambda)|\geq \eta_2$ and $\Re(\lambda) \neq 0$. To this end, we observe that the resolvent problem~\eqref{resolvent_problem} has a unique solution $\ub \in H^2(\R)$ if and only if the conjugate problem
\begin{align}\label{resolvent_problem_complex}
	(\widetilde{\mathcal{L}}(\ubu) -\lambda) \widetilde{\ub} = \widetilde{\mathbf{w}}
\end{align}
posses a unique solution $\widetilde{\ub} \in H^2(\R)$, where we denote $\widetilde{\mathcal{L}}(\ubu) = \widetilde{J} \widetilde{L}(\ubu)$ with $\widetilde{J} = S^{-1} J S$, $\widetilde{L}(\ubu) = S^{-1} L(\ubu) S$, $\widetilde{\ub} = S^{-1} \ub$, $\widetilde{\mathbf{w}} = S^{-1} \mathbf{w}$ and
\begin{align*}
    S = \frac{1}{2}
    \begin{pmatrix}
        1 & 1\\ -\iu & \iu
    \end{pmatrix} \in \C^{2\times 2}.
\end{align*}
We note that conjugation with the matrix $S$ corresponds to a coordinate transform transferring the formulation~\eqref{LLE_system} of~\eqref{LLE} as a system in $(\Re(u),\Im(u))^\top$ into its formulation as a system in $(u,\overline{u})^\top$. One readily computes
\begin{align*}
		\widetilde{J} = 
		\begin{pmatrix}
			-\iu & 0 \\ 0 & \iu
		\end{pmatrix}, \qquad
		\widetilde{L}(\ubu) = 
		\begin{pmatrix}
			-\partial_x^2 + \zeta - 2 |\tilde{\ubu}|^2 & - \tilde{\ubu}^2 \\
			-\overline{\tilde{\ubu}}^2 & -\partial_x^2 + \zeta - 2 |\tilde{\ubu}|^2 
		\end{pmatrix},
\end{align*}
where we denote $\tilde{\ubu} = \ubu_1 + \iu \ubu_2$. We define the operators $A\colon H^2(\R) \subset L^2(\R) \to L^2(\R)$ and $B\colon L^2(\R)\to L^2(\R)$ by $A u = -u'' + (\zeta - 2 |\tilde{\ubu}|^2) u$ and $Bu = - \tilde{\ubu}^2 u$, respectively. Then, using integration by parts we infer that the self-adjoint operator $A$ can be bounded from below as
\begin{align*}
    \langle A u,u \rangle_{L^2} \geq \|u'\|_{L^2}^2 + (\zeta - 2 \|\tilde{\ubu}\|_{L^\infty}^2) \|u\|_{L^2}^2
    \geq (\zeta- 2 \rho^2) \|u\|_{L^2}^2
\end{align*}
for all $u \in H^2(\R)$. We note that the lower bound of $A$ depends on $\rho$ and $\zeta$ only. Clearly, $B$ is a bounded operator with $\|B\|_{L^2 \to L^2} \leq \rho^2$. Employing~\cite[Theorem 4]{Bengel2024Stability}, we find constants $C_2,\eta_2 >0$, depending on $\rho$ and $\zeta$ only, such that for every $\widetilde{\mathbf{w}} \in L^2(\R)$ the problem~\eqref{resolvent_problem_complex} has a unique solution $\widetilde{\ub} \in H^2(\R)$ with 
\begin{align*}\|\widetilde{\ub}\|_{L^2} \leq C_2 \frac{\|\widetilde{\mathbf{w}}\|_{L^2}}{|\Re(\lambda)|},\end{align*} 
provided that $\Re(\lambda) \neq 0$ and $|\Im(\lambda)| \geq \eta_2.$ Thus, we have proved
\begin{align*}
    \{\lambda \in \C : |\Im(\lambda)| \geq \eta_2, \Re(\lambda) \neq 0\} \cap \sigma(\mathcal{L}(\ubu)) = \emptyset
\end{align*}
and the claim follows.
\end{proof}

\subsection{Stability of \texorpdfstring{$N$-pulse}{N-pulse} solutions}\label{ssec:stab_N-pulse}

We determine the spectral stability of the $N$-pulse solutions established in Theorem~\ref{thm:ex_N-pulse}. If one of the primary $1$-pulse solutions constituting the $N$-pulse is spectrally unstable, spectral instability of the $N$-pulse follows from~\cite{BengeldeRijk}. On the other hand, if the $N$-pulse is comprised of $N$ spectrally stable $1$-pulses, then we employ Lemma~\ref{lem:a-priori_spectrum} and~\cite[Lemma~3.3]{BengeldeRijk}, to confine the spectral stability problem to a small ball $B_{\delta_0}(0)$ centered at the origin. Spectral stability is then decided by the $N$ small eigenvalues in $B_{\delta_0}(0)$, whose position with respect to the imaginary axis is described by Theorem~\ref{thm:ex_N-pulse}. All in all, we arrive at the following result. 

\begin{Theorem}\label{thm:stab_N-pulse}
Let $N \in \N$ and $\ell \in \{0,\ldots,N-1\}$. Set $n = {\lfloor \frac{N}{2} \rfloor}$ and $\alpha_0 = N\! \mod 2 \in \{0,1\}$. Let $\ubu_1,\ldots,\ubu_{n+\alpha_0}$ be a sequence of stationary $1$-pulse solutions to~\eqref{LLE_system}, established in Theorem~\ref{thm:1-pulse}. For each $\mathbf{k} \in \N^n$, we denote by $\ubu_{\mathbf{k},\ell}$ the associated $N$-pulse solutions to~\eqref{LLE_system}, established in Theorem~\ref{thm:ex_N-pulse}.

The following assertions hold. 
\begin{itemize}
    \item[(i)] If $\ubu_1,\ldots,\ubu_{n+\alpha_0}$ are spectrally stable solutions of~\eqref{LLE_system}, then there exists a constant $k_0 \in \N$ such that, whenever $\mathbf{k} = (k_1,\ldots,k_n) \in \N^n$ satisfies $\min\{k_1,\ldots,k_n\} \geq k_0$, the $N$-pulse solution $\ubu_{\mathbf{k},\ell}$ of~\eqref{LLE_system} is spectrally stable if $\ell = N-1$ and spectrally unstable if $\ell \neq N-1$.
    \item[(ii)] Assume that there exists $i \in \{1,\ldots,n+\alpha_0\}$ such that $\ubu_i$ is a spectrally unstable solution of~\eqref{LLE_system}, then there exists a constant $k_0 \in \N$ such that, whenever $\mathbf{k} = (k_1,\ldots,k_n) \in \N^n$ satisfies $\min\{k_1,\ldots,k_n\} \geq k_0$, the $N$-pulse solution $\ubu_{\mathbf{k},\ell}$ of~\eqref{LLE_system} is spectrally unstable.
\end{itemize}
\end{Theorem}

\begin{proof}
We start with the proof of the first assertion. Let $\delta_0 > 0$ be as in Theorem~\ref{thm:ex_N-pulse}. Using the bound~\eqref{boundmultipulse} in Theorem~\ref{thm:ex_N-pulse} and applying Lemma~\ref{lem:a-priori_spectrum}, we find $\mathbf{k}$-independent constants $\eta_1,\eta_2>0$ such that
\begin{align*}
    \sigma(\El(\ubu_{\mathbf{k},\ell})-\eps) \cap \{\lambda \in \C : \Re(\lambda) \geq -\tfrac{\eps}{2}\} \subset \{\lambda \in \C : - \tfrac{\eps}{2} \leq \Re(\lambda) \leq \eta_1, |\Im(\lambda)| \leq \eta_2\}.
\end{align*}
On the other hand, by the spectral stability of $\ubu_i$, there exists a $\mathbf{k}$-independent constant $\tau > 0$ such that
\begin{align*}
    \sigma(\El(\ubu_i)-\eps) \cap \{\lambda \in \C : \Re(\lambda) > -\tau\} = \{0\}
\end{align*}
for $i = 1,\ldots,n+\alpha_0$. Hence,~\cite[Lemma~3.3]{BengeldeRijk} yields a constant $k_0 \in \N$ such that, provided $\mathbf{k} = (k_1,\ldots,k_n) \in \N^n$ satisfies $\min\{k_1,\ldots,k_n\} \geq k_0$, the compact set
\begin{align*}
\{\lambda \in \C : -\tau \leq \Re(\lambda) \leq \eta_1, |\Im(\lambda)| \leq \eta_2\} \setminus B_{\delta_0}(0)
\end{align*}
lies in the resolvent set of $\El(\ubu_{\mathbf{k},\ell})-\eps$. We conclude that all spectrum of $\El(\ubu_{\mathbf{k},\ell})-\eps$ in the half plane $\{\lambda \in \C : \Re(\lambda) \geq -\tau\}$ must be confined to the ball $B_{\delta_0}(0)$. The spectrum of $\El(\ubu_{\mathbf{k},\ell})-\eps$ in the ball $B_{\delta_0}(0)$ consists of precisely $\ell$ eigenvalues of negative real part, $N-1-\ell$ eigenvalues of positive real part and one algebraically simple eigenvalue $0$ by Theorem~\ref{thm:ex_N-pulse} (all counted with algebraic multiplicities). This yields the first assertion. 

For the second assertion, we observe that if $\ubu_i$ is spectrally unstable, then there exists $\lambda_0 \in \sigma(\El(\ubu_i) - \eps)$ with $\Re(\lambda_0) > 0$. By~\cite[Lemma~4]{Bengel2024Stability} the essential spectrum of $\El(\ubu_j) - \eps$ lies on the line $\{\lambda \in \C : \Re(\lambda) = -\eps\}$ for $j = 1,\ldots,n+\alpha_0$. Combining the last two lines with~\cite[Theorem~6.2]{BengeldeRijk} yields a constant $k_0 \in \N$ such that, provided $\mathbf{k} = (k_1,\ldots,k_n) \in \N^n$ satisfies $\min\{k_1,\ldots,k_n\} \geq k_0$, $\El(\ubu_{\mathbf{k},\ell}) - \eps$ possesses an eigenvalue $\lambda$ of real part $\Re(\lambda) \geq \Re(\lambda_0)/2 > 0$, which proves the second assertion. 
\end{proof}

\begin{Remark} \label{rem:short_time_instability}
As mentioned in Remark~\ref{rem:long_term_instability}, the $N$-pulses, established in Theorem~\ref{thm:ex_N-pulse}, can suffer from long-time instabilities triggered by the $N$  eigenvalues near the origin, which are exponentially small with respect to the distances between pulses.  On the other hand, if one of the primary 1-pulses constituting the $N$-pulse is spectrally unstable, the proof of Theorem~\ref{thm:stab_N-pulse} shows that we find an eigenvalue $\lambda$ of the linearization of~\eqref{LLE_system} about the $N$-pulse of real part $\Re(\lambda) \geq \Re(\lambda_0)/2 > 0$. That is, the real part of this eigenvalue obeys a positive lower bound, which is independent of the distances between the $1$-pulses constituting the $N$-pulse. Therefore, the instability in Theorem~\ref{thm:stab_N-pulse}.(ii) can be interpreted as a \emph{short-time instability}. 
\end{Remark}

The asymptotic orbital stability of spectrally stable $N$-pulses is now a direct consequence of the following statement, which was proved in~\cite{Bengel2024Stability}. 

\begin{Theorem}[\!{\!\cite[Theorem 3]{Bengel2024Stability}}] \label{thm:nonstab_Npulse}
Let $\ub_\infty \in \R^2$ and let $\ubu \colon \R \to \R^2$ be a smooth stationary solution to~\eqref{LLE_system} such that $\ubu(x)$ converges to $\ub_\infty$ as $x \to \pm \infty$. If $\ubu$ is spectrally stable, then there exist constants $C,\delta,\eta > 0$ such that for all $\mathbf{v}_0 \in H^1(\R)$ with $\|\mathbf{v}_0\|_{H^1} \leq \delta$ there exist a constant $\gamma \in \R$ and a function $\mathbf{v} \in C\big([0,\infty),H^1(\R)\big)$ with $\mathbf{v}(0) = \mathbf{v}_0$ such that $\ub = \ubu + \mathbf{v}$ is the unique global mild solution of~\eqref{LLE_system} with $\ub(0) = \ubu + \mathbf{v}_0$ enjoying the estimate
\begin{align*}
\left\|\ub - \ubu(\cdot + \gamma)\right\|_{H^1} \leq C\eu^{-\eta t} \delta
\end{align*}
for all $t \geq 0$. 
\end{Theorem}

\subsection{Stability of periodic \texorpdfstring{$N$-pulse}{N-pulse} solutions}

We now turn to the stability analysis of the periodic multipulse solutions to~\eqref{LLE_system}, whose existence, as outlined in~\S\ref{ssec:ex_N-pulse}, follows by combining Theorems~\ref{thm:ex_N-pulse} and~\ref{thm:ex_periodic}. Our first step is to establish \emph{diffusive spectral stability} of these solutions, which, in turn, yields their nonlinear stability against localized perturbations and against subharmonic perturbations of any wavelength.\footnote{We note that diffusive spectral stability is a standard assumption in the nonlinear stability analysis of periodic traveling or steady waves in dissipative systems, see~\cite{Haragus2023Nonlinear} and further references therein.} 

Diffusive spectral stability is defined in terms of the $1$-parameter family of Bloch operators associated with the linearization $\El(\ubu) - \eps$ of system~\eqref{LLE_system} about a periodic smooth stationary solution $\ubu = (\ubu_1,\ubu_2)^\top$ with period $T > 0$. These Bloch operators $\El_\xi(\ubu)-\eps \colon H_\per^2(0,T) \subset L_\per^2(0,T) \to L_\per^2(0,T)$ are given by 
\begin{align*}
    L_\xi(\ubu)= - (\partial_x + \iu \xi T)^2 +\zeta - 
	\begin{pmatrix}
		3 \ub_{1}^2 + \ub_{2}^2 & 2 \ub_{1} \ub_{2} \\
		2 \ub_{1} \ub_{2} &  \ub_{1}^2 + 3\ub_{2}^2
	\end{pmatrix}, \qquad \mathcal{L}_\xi(\ubu) = JL_\xi(\ubu)
\end{align*}
with $\xi \in [-\pi, \pi)$. We then have the well-known spectral decomposition
\begin{align} \label{characterization_periodic_spectrum}
    \sigma(\El(\ubu)-\eps ) = \bigcup_{\xi \in \left[-\pi,\pi\right)} \sigma\left(\El_\xi(\ubu)-\eps\right),
\end{align}
cf.~\cite{Gardner1993Structure}. The definition of diffusive spectral stability now reads as follows. 

\begin{Definition} \label{def:diffusive_spectral_stability}
A smooth stationary $T$-periodic solution $\ubu \colon \R \to \R^2$ of~\eqref{LLE_system} is \emph{diffusively spectrally stable} provided the following conditions hold:
\begin{itemize}
    \item[(i)] We have $\sigma(\El(\ubu)-\eps ) \subset \{\lambda \in \C : \Re(\lambda) < 0\} \cup \{0\}$;
    \item[(ii)] There exists $\vartheta>0$ such that for all $\xi \in [-\pi,\pi)$ we have $\Re\left(\sigma\left(\El_\xi(\ubu)-\eps\right)\right) \leq - \vartheta \xi^2$;
    \item[(iii)] $0$ is a simple eigenvalue of the Bloch operator $\El_0(\ubu)-\eps$.
\end{itemize}
\end{Definition}

Diffusive spectral stability of the periodic multipulse solutions to~\eqref{LLE_system}, studied in this paper, can be obtained if the underlying multipulse is spectrally stable. In this case, we can apply the a-priori bounds in Lemma~\ref{lem:a-priori_spectrum} and the results in~\cite{BengeldeRijk,Gardner1997Spectral}, to preclude any unstable spectrum outside a small ball $B_{\varrho}(0)$ of radius $\varrho > 0$ centered at the origin. By~\cite{BengeldeRijk,Gardner1997Spectral} the spectrum inside the ball $B_{\varrho}(0)$ is given by a single smooth curve $\{\lambda_0(\xi) : \xi \in [-\pi,\pi)\}$, which touches the origin by translational invariance. Leading-order control on this critical spectral curve, provided by~\cite{Sandstede2001Stability}, then yields constants $\widetilde{T}_0,\varpi > 0$ and a partitioning of the half line $[\widetilde{T}_0,\infty)$ in intervals $I_j = (\widetilde{T}_0 + j \varpi, \widetilde{T}_0 + (j+1) \varpi)$, $j \in \N$ such that diffusive spectral stability holds if the period $T$ lies in $I_j$ with $j$ even, whereas spectral instability holds for $T \in I_j$ if $j$ is odd. That is, the stability of the multipulse train alternates with the period. On the other hand, if the underlying multipulse solution is spectrally unstable, then the associated multipulse trains are spectrally unstable by~\cite{BengeldeRijk,Gardner1997Spectral} for \emph{all} periods $T > 0$ sufficiently large. 

\begin{Theorem}\label{thm:stab_periodic_pulse}
Let \begin{align*}\un{U} = (\un{U}_1,\un{U}_2,\un{U}_3,\un{U}_4)^\top \colon \R \to \R^4\end{align*} be a nondegenerate symmetric homoclinic solution of~\eqref{LLE_DS} connecting to a saddle-focus equilibrium $U_\infty \in \R^4$, as established in Theorem~\ref{thm:ex_N-pulse}. Let $\{\un{U}_T\}_{T \geq T_0}$ be the corresponding family of smooth $T$-periodic symmetric solutions 
\begin{align*}\un{U}_T = (\un{U}_{T,1},\un{U}_{T,2},\un{U}_{T,3},\un{U}_{T,4})^\top \colon \R \to \R^4\end{align*} of~\eqref{LLE_DS}, established in Theorem~\ref{thm:ex_periodic}. Denote by $\ubu = (\un{U}_1,\un{U}_2)^\top, \ubu_T = (\un{U}_{T,1},\un{U}_{T,2})^\top \colon \R \to \R^2$ the associated stationary solutions of~\eqref{LLE_system}.

The following assertions hold.
\begin{itemize}
\item[(i)] If $\ubu$ is a spectrally stable pulse solution of~\eqref{LLE_system}, then there exist constants $\widetilde{T}_0,C,\tau,\delta,a >0$ and $b \in \R$, an open set $U \subset \C$ containing the real interval $[-\pi,\pi]$ and an analytic map $\lambda_0 \colon U \to \C$ such that for all $T \geq \widetilde{T}_0$ we have
\begin{align} \label{period_spec_bound3}
    \sigma(\mathcal{L}(\ubu_T) - \eps) \cap \{\lambda\in\C : \Re(\lambda) \geq -\tau\} = \{\lambda_0(\xi) : \xi\in [-\pi,\pi) \}.
\end{align}
Moreover, $\lambda_0(\xi) = \lambda_0(-\xi)$ is a real-valued algebraically simple eigenvalue of the Bloch operator $\El_\xi(\ubu_T) - \eps$ for each $\xi \in [-\pi,\pi)$ and $T \geq \widetilde{T}_0$. Finally, the estimate
\begin{align} \label{approx_criticalcurve}
\left|\lambda_0(\xi) - a (\cos(\xi)-1) \eu^{-2 \alpha T } \sin(2\beta T +b)\right| \leq C |\eu^{\iu \xi} - 1| \eu^{-(2\alpha+\delta)T}
\end{align}
holds for all $\xi \in U$ and $T \geq \widetilde{T}_0$, where $\alpha,\beta > 0$ are as in Proposition~\ref{prop:saddle_focus}.

If $T \geq \widetilde{T}_0$ is such that $\sin(2\beta T+b)>0$, then $\ubu_T$ is diffusively spectrally stable as a periodic stationary solution to~\eqref{LLE_system}. Moreover, if $T \geq \widetilde{T}_0$ is such that $\sin(2\beta T+b)<0$, then $\ubu_T$ is spectrally unstable as a stationary solution to~\eqref{LLE_system}.
\item[(ii)] Assume that $\ubu$ is spectrally unstable. Then, there exists a constant $\widetilde{T}_0 > 0$ such that for all $T \geq \widetilde{T}_0$ the stationary solution $\ubu_T$ of~\eqref{LLE_system} is spectrally unstable. In particular, $\El_\xi(\ubu_T) - \eps$ possesses spectrum in the open right-half plane for all $\xi \in [-\pi,\pi)$.
\end{itemize}
\end{Theorem}

\begin{proof}
We start with the proof of the first assertion. Since the stationary pulse solution $\ubu$ of~\eqref{LLE_system} is spectrally stable, there exists a constant $\tau_0>0$ such that
\begin{align} \label{specstab_pulse}
    \sigma(\El(\ubu)-\eps) \cap \{\lambda \in \C : \Re(\lambda)\geq - \tau_0\} = \{0\},
\end{align}
and $0$ is an algebraically simple eigenvalue of $\El(\ubu) - \eps$. Moreover, by the bound~\eqref{eq:limit_periodic_to_homoclinic} in Theorem~\ref{thm:ex_periodic}, the $L^\infty$-norm of the solution $\ubu_T$ can be bounded by a $T$-independent constant for $T \geq T_0$. Therefore, Lemma~\ref{lem:a-priori_spectrum}, yields a $T$-independent compact set $\mathcal{K}\subset \C$ such that
\begin{align} \label{period_spec_bound2}
    \sigma(\El(\ubu_T)-\eps) \cap \{\lambda \in \C : \Re(\lambda) \geq -\tfrac{\eps}{2}\} \subset \mathcal{K}.
\end{align}
for all $T \geq T_0$. Take $\varrho \in (0,\min\{\frac{\eps}{2},\tau_0\})$ and set $\tilde{\mathcal{K}} := \{\lambda \in \mathcal{K} : \Re(\lambda) \geq -\tau_0\} \setminus B_\varrho(0)$. Using that $\tilde{\mathcal{K}}$ is compact and~\eqref{specstab_pulse} holds,~\cite[Lemma~4.2]{BengeldeRijk} yields, provided $T \geq T_0$ is sufficiently large, that
\begin{align} \label{periodic_spec_bound}
     \sigma(\El(\ubu_T)-\eps) \cap \tilde{\mathcal{K}} = \emptyset
\end{align}
and, moreover,
\begin{align} \label{bloch_spec_bound}
\sigma(\El_\xi(\ubu_T)-\eps) \cap B_\varrho(0) = \{\lambda_0(\xi)\},
\end{align}
where $\lambda_0(\xi)$ is an algebraically simple eigenvalue of the Bloch operator $\El_\xi(\ubu_T) - \eps$ for each $\xi \in [-\pi,\pi)$. Hence, using that the Bloch operators depend analytically on $\xi$, it follows from standard analytic perturbation theory, see~\cite[Sections~II.1 and~VII.3]{Kato1995Perturbation}, that there exists an open neighborhood $U \subset \C$ of the real interval $[-\pi,\pi]$ such that $\lambda_0(\xi)$ can be extended to an analytic map $\lambda_0 \colon U \to \C$. By the same dimension-counting argument as in the proof of Theorem~\ref{thm:ex_N-pulse}, the (un)stable manifolds along the homoclinic $\un{U}$ connecting to the saddle-focus equilibrium $U_\infty$ in the dynamical system~\eqref{LLE_DS} are not in an inclination-flip configuration. Therefore, provided $T > 0$ is sufficiently large,~\cite[Theorem 5.6]{Sandstede2001Stability} yields that $\lambda_0(\xi)$ is real-valued, we have $\lambda_0(\xi) = \lambda_0(-\xi)$ for all $\xi\in [-\pi,\pi)$ and the approximation~\eqref{approx_criticalcurve} holds with $a > 0$ (after possibly shifting $b \mapsto b + \pi$). Finally, the identities~\eqref{characterization_periodic_spectrum},~\eqref{period_spec_bound2},~\eqref{periodic_spec_bound} and~\eqref{bloch_spec_bound} imply~\eqref{period_spec_bound3} with $\tau = \min\{\tau_0,\eps/2\}>0$.

If $T > 0$ is sufficiently large with $\sin(2\beta T + b) > 0$, then Cauchy's estimate in conjunction with the bound~\eqref{approx_criticalcurve} yield $\lambda_0''(0) < 0$. Hence, combining the latter with~\eqref{approx_criticalcurve} and $\lambda_0(0) = 0 = \lambda_0'(0)$, we infer that, provided $T > 0$ is sufficiently large, there exists $\vartheta > 0$ such that $\lambda_0(\xi) \leq -\vartheta \xi^2$ for all $\xi \in [-\xi,\xi]$. This, together with~\eqref{period_spec_bound3}, implies diffusive spectral stability of $\ubu_T$ as a periodic stationary solution to~\eqref{LLE_system}. On the other hand, if $T > 0$ is sufficiently large with $\sin(2\beta T + b) < 0$, then by estimate~\eqref{approx_criticalcurve} there exists $\xi \in [-\pi,\pi) \setminus \{0\}$ such that $\lambda_0(\xi) > 0$. Therefore, $\ubu_T$ is spectrally unstable as a stationary solution to~\eqref{LLE_system} upon recalling~\eqref{period_spec_bound3}. This finishes the proof of the first assertion.

We proceed with proving the second assertion. Assume that $\ubu$ is spectrally unstable. Then, there exists $\lambda_0 \in \sigma(\El(\ubu) - \eps)$ with $\Re(\lambda_0) > 0$. By~\cite[Lemma~4]{Bengel2024Stability} the essential spectrum of $\El(\ubu) - \eps$ is confined to the line $\{\lambda \in \C : \Re(\lambda) = -\eps\}$. Therefore, we can apply~\cite[Theorem~7.2]{BengeldeRijk} to yield that, provided $T > 0$ is sufficiently large, $\sigma(\El_\xi(\ubu_T) - \eps)$ possesses for each $\xi \in [-\pi,\pi)$ an element $\lambda_*(\xi)$ of real part $\Re(\lambda_*(\xi)) \geq \Re(\lambda_0)/2 > 0$, which proves the second assertion. 
\end{proof}

\begin{Remark}
If the underlying multipulse solution is spectrally unstable, then Theorem~\ref{thm:stab_periodic_pulse}.(ii) shows that each of the Bloch operators $\El_\xi(\ubu_T) - \eps$ with $\xi \in [-\pi,\pi)$ possesses spectrum in the open right-half plane. That is, the multipulse train $\ubu_T$ is spectrally unstable against subharmonic perturbations of any wavelength. In particular, it is spectrally unstable against co-periodic perturbations. On the other hand, if the underlying multipulse solution is spectrally stable, then Theorem~\ref{thm:stab_periodic_pulse}.(i) yields that the spectrum of the Bloch operator $\El_0(\ubu_T) - \eps$ lies in the open left-half plane except for the algebraically simple eigenvalue $\lambda_0(0) = 0$. That is, the multipulse train $\ubu_T$ is spectrally stable against co-periodic perturbations. Moreover, if in addition $\sin(2\beta T + b) < 0$ holds, we find $\lambda_0''(0) > 0$ by applying Cauchy's estimate to the bound~\eqref{approx_criticalcurve}. Hence, using that $\lambda_0(0), \lambda_0'(0) = 0$ holds by Theorem~\ref{thm:stab_periodic_pulse}, we infer that $\El_\xi(\ubu_T) - \eps$ possesses unstable spectrum for each $\xi \in [-\pi,\pi) \setminus \{0\}$ sufficiently small. We conclude that the multipulse train is sideband unstable, i.e., it is spectrally unstable against subharmonic perturbations of sufficiently large wavelength.
\end{Remark}

\begin{Remark}
Similarly as in the case for multipulse solutions, cf.~Remark~\ref{rem:short_time_instability}, we can distinguish between short- and long-time instabilities of multipulse trains. On the one hand, the critical spectral curve $\lambda_0(\xi)$ in Theorem~\ref{thm:stab_periodic_pulse}.(i) is exponentially small in terms of the period $T$ by estimate~\eqref{approx_criticalcurve}. Thus, any instabilities arising from this curve can be interpreted as long-time instabilities of the multipulse train. On the other hand, if the underlying multipulse solution is spectrally unstable, then the proof of Theorem~\ref{thm:stab_periodic_pulse}.(ii) shows that the linearization of~\eqref{LLE_system} about the multipulse train admits unstable spectrum, whose real part can be bounded from below by a positive bound, independent of $T$. Thus, these instabilities can be interpreted as short-time instabilities. 
\end{Remark}

The proof of our main result, Theorem~\ref{t:mainresult}, now follows by combining the diffusive spectral stability result in Theorem~\ref{thm:stab_periodic_pulse} with~\cite{Haragus2023Nonlinear,Stanislavova2018Asymptotic}. 

\begin{proof}[Proof of Theorem~\ref{t:mainresult}]
Since it holds $8 \zeta < \pi^2 f^2$, $\pi f \cos \theta_0 = 2 \sqrt{2 \zeta}$ and $\sin \theta_0 > 0$, Theorem~\ref{thm:1-pulse} provides constants $C_0,\eps_0 > 0$ such that for each $\eps \in (0,\eps_0)$ there exist an asymptotic state $\ub_{\infty,\eps}$ and an even spectrally stable smooth stationary $1$-pulse solution $\ubu_\eps$ of~\eqref{LLE_system} obeying~\eqref{estimates_1_pulse}. 

Fix $\eps \in (0,\eps_0)$. We apply Theorems~\ref{thm:ex_N-pulse} and~\ref{thm:stab_N-pulse} with $\ell = N-1$ to yield constants $C_1,k_0 > 0$ such that for each $\mathbf{k} = (k_1,\ldots,k_n) \in \N^n$ with $\min\{k_1,\ldots,k_n\} \geq k_0$ there exist distances $T_{i,\eps}^{k_1},\ldots,T_{n,\eps}^{k_n}$ and an even spectrally stable smooth stationary $N$-pulse solution $\ubu_{\mathbf{k},\eps}$ of~\eqref{LLE_system} enjoying the estimate
\begin{align*}
\begin{split}
&\left|\ubu_{\mathbf{k},\eps}(x) - \ub_{\infty,\eps} - \alpha_0 \left(\ubu_\eps(x) - \ub_{\infty,\eps}\right) - \sum_{i = 1}^n \left(\ubu_\eps\left(x - T_{1,\eps}^{k_1} - \ldots - T_{i,\eps}^{k_i}\right)\right. \right.\\ 
&\qquad \qquad \quad \qquad \qquad \quad \qquad \left.  \phantom{ \sum_{i = 1}^n }\left. +\, \ubu_\eps\left(x + T_{1,\eps}^{k_1} + \ldots + T_{i,\eps}^{k_i}\right) - 2\ub_{\infty,\eps}\right) \right| \leq \frac{C_1}{\min\{k_1,\ldots,k_n\}}
\end{split}
\end{align*}
for $x \in \R$. Here, the sequence $\{T_{i,\eps}^{k}\}_k$ of pulse distances is monotonically increasing with $T_{i,\eps}^k \to \infty$ as $k \to \infty$ for $i = 1,\ldots,n$. Combining the latter estimate with~\eqref{estimates_1_pulse} implies
\begin{align} \label{boundmultipulse3}
\begin{split}
&\left|\ubu_{\mathbf{k},\eps}(x) - \alpha_0\boldsymbol{\phi}_{\theta_0}(x) - \sum_{i = 1}^n \left(\boldsymbol{\phi}_{\theta_0}\left(x - T_{1,\eps}^{k_1} - \ldots - T_{i,\eps}^{k_i}\right) +  \boldsymbol{\phi}_{\theta_0}\left(x + T_{1,\eps}^{k_1} + \ldots + T_{i,\eps}^{k_i}\right) \right) \right|\\ 
&\qquad \leq 2C_0\eps 
\end{split}
\end{align}
for each $x \in \R$ and $\mathbf{k} \in \N^n$ with $\min\{k_1,\ldots,k_n\} \geq k_0$, 
upon taking $k_0 > 0$ larger if necessary.

Now, fix $\mathbf{k} \in \N^n$ with $\min\{k_1,\ldots,k_n\} \geq k_0$. By Proposition~\ref{prop:saddle_focus} and Theorems~\ref{thm:ex_periodic} and~\ref{thm:stab_periodic_pulse} there exists a monotonically increasing sequence of periods $\{L_{\mathbf{k},\eps}^m\}_m$ such that for each $m \in \N$ there exists an even diffusively spectrally stable smooth stationary periodic solution $\ubu_{m,\mathbf{k},\eps}(x)$ of~\eqref{LLE_system} of period $L_{\mathbf{k},\eps}^m$ satisfying the estimate
\begin{align} \label{boundperiodicpulse2}
\sup_{x \in \left[-\frac12 L_{\mathbf{k},\eps}^m, \frac12 L_{\mathbf{k},\eps}^m\right]} \left|\ubu_{m,\mathbf{k},\eps}(x) -  \ubu_{\mathbf{k},\eps}(x)\right| \leq C_0 \eps.
\end{align}
Here, the sequence $\{L_{\mathbf{k},\eps}^m\}_m$ tends to $\infty$ as $m \to \infty$ and obeys~\eqref{sum_period} for each $m \in \N$. Thus, we have established assertions~(i),~(ii) and~(iv). Moreover, assertion~(iii) follows readily by combing~\eqref{boundmultipulse3} and~\eqref{boundperiodicpulse2}. Assertion~(v) is a direct consequence of the diffusive spectral stability of $\ubu_{m,\mathbf{k},\eps}$ in combination with~\cite[Theorem~1]{Stanislavova2018Asymptotic}, see also~\cite[Theorem~1.2]{Haragus2024Nonlinear}. Similarly, assertion~(vi) follows immediately from~\cite[Theorem~1.3]{Haragus2023Nonlinear}. 
\end{proof}

\begin{Remark}
Diffusive spectral stability of $T$-periodic stationary solutions to the LLE even yields a nonlinear stability result~\cite{Haragus2024Nonlinear} against $MT$-periodic perturbations that is uniform in $M \in \mathbb{N}$, as well as nonlinear stability against nonlocalized phase modulations~\cite{Zumbrun2024Forward}. We refer to~\cite{Haragus2024Nonlinear,Zumbrun2024Forward} for further details.
\end{Remark}

\bibliographystyle{abbrv}
\bibliography{bibliography}

\begin{thebibliography}{10}

\bibitem{Bengel2024Stability}
L.~Bengel.
\newblock Stability of solitary wave solutions in the {L}ugiato-{L}efever
  equation.
\newblock {\em Z. Angew. Math. Phys.}, 75(4):Paper No. 130, 2024.

\bibitem{BengeldeRijk}
L.~Bengel and B.~de~Rijk.
\newblock Multiple front and pulse solutions in spatially periodic systems.
\newblock {\em arXiv:2502.02467}, 2025.

\bibitem{Brasch2016Photonic}
V.~Brasch, M.~Geiselmann, T.~Herr, G.~Lihachev, M.~H.~P. Pfeiffer, M.~L.
  Gorodetsky, and T.~J. Kippenberg.
\newblock Photonic chip–based optical frequency comb using soliton
  {C}herenkov radiation.
\newblock {\em Science}, 351(6271):357--360, 2016.

\bibitem{Cazenave1982Orbital}
T.~Cazenave and P.-L. Lions.
\newblock Orbital stability of standing waves for some nonlinear
  {S}chr\"{o}dinger equations.
\newblock {\em Comm. Math. Phys.}, 85(4):549--561, 1982.

\bibitem{Champneys1994Homoclinic}
A.~R. Champneys.
\newblock Subsidiary homoclinic orbits to a saddle-focus for reversible
  systems.
\newblock {\em Internat. J. Bifur. Chaos Appl. Sci. Engrg.}, 4(6):1447--1482,
  1994.

\bibitem{Chembo2017}
Y.~K. Chembo, D.~Gomila, M.~Tlidi, and C.~R. Menyuk.
\newblock Topical issue: theory and applications of the {L}ugiato-{L}efever
  equation.
\newblock {\em Eur. Phys. J. D}, 71, 2017.

\bibitem{Coddington2008}
I.~Coddington, W.~C. Swann, and N.~R. Newbury.
\newblock Coherent multiheterodyne spectroscopy using stabilized optical
  frequency combs.
\newblock {\em Phys. Rev. Lett.}, 100:013902, Jan 2008.

\bibitem{Delcey2018Instability}
L.~Delcey and M.~Haragus.
\newblock Instabilities of periodic waves for the {L}ugiato-{L}efever equation.
\newblock {\em Rev. Roumaine Math. Pures Appl.}, 63(4):377--399, 2018.

\bibitem{Delcey2018Periodic}
L.~Delcey and M.~Haragus.
\newblock Periodic waves of the {L}ugiato-{L}efever equation at the onset of
  {T}uring instability.
\newblock {\em Philos. Trans. Roy. Soc. A}, 376(2117):20170188, 21, 2018.

\bibitem{DelHaye2007}
P.~Del'Haye, A.~Schliesser, O.~Arcizet, T.~Wilken, R.~Holzwarth, and T.~J.
  Kippenberg.
\newblock Optical frequency comb generation from a monolithic microresonator.
\newblock {\em Nature}, 450(7173):1214--1217, Dec 2007.

\bibitem{Doelman2009}
A.~Doelman, B.~Sandstede, A.~Scheel, and G.~Schneider.
\newblock The dynamics of modulated wave trains.
\newblock {\em Mem. Amer. Math. Soc.}, 199(934):viii+105, 2009.

\bibitem{Engel2000One}
K.-J. Engel and R.~Nagel.
\newblock {\em One-parameter semigroups for linear evolution equations}, volume
  194 of {\em Graduate Texts in Mathematics}.
\newblock Springer-Verlag, New York, 2000.
\newblock With contributions by S. Brendle, M. Campiti, T. Hahn, G. Metafune,
  G. Nickel, D. Pallara, C. Perazzoli, A. Rhandi, S. Romanelli and R.
  Schnaubelt.

\bibitem{Ferdous2011}
F.~Ferdous, H.~Miao, D.~E. Leaird, K.~Srinivasan, J.~Wang, L.~Chen, L.~T.
  Varghese, and A.~M. Weiner.
\newblock Spectral line-by-line pulse shaping of on-chip microresonator
  frequency combs.
\newblock {\em Nature Photonics}, 5(12):770--776, Dec 2011.

\bibitem{Gardner1993Structure}
R.~A. Gardner.
\newblock On the structure of the spectra of periodic travelling waves.
\newblock {\em J. Math. Pures Appl. (9)}, 72(5):415--439, 1993.

\bibitem{Gardner1997Spectral}
R.~A. Gardner.
\newblock Spectral analysis of long wavelength periodic waves and applications.
\newblock {\em J. Reine Angew. Math.}, 491:149--181, 1997.

\bibitem{Gaertner2020Soliton}
J.~G{\"a}rtner and W.~Reichel.
\newblock Soliton solutions for the {L}ugiato-{L}efever equation by analytical
  and numerical continuation methods.
\newblock In {\em Mathematics of Wave Phenomena}, pages 179--195, Cham, 2020.
  Springer International Publishing.

\bibitem{Gaertner2019Bandwidth}
J.~G\"artner, P.~Trocha, R.~Mandel, C.~Koos, T.~Jahnke, and W.~Reichel.
\newblock Bandwidth and conversion efficiency analysis of dissipative kerr
  soliton frequency combs based on bifurcation theory.
\newblock {\em Phys. Rev. A}, 100:033819, Sep 2019.

\bibitem{Gasmi2022Lugiato}
E.~Gasmi.
\newblock {\em On the {L}ugiato-{L}efever model for frequency combs in a
  dual-pumped ring resonator: with an appendix on band structures for periodic
  fractional {S}chr{\"o}dinger operators}.
\newblock PhD thesis, Karlsruher Institut f{\"u}r Technologie (KIT), 2022.

\bibitem{Godey2017Bifurcation}
C.~Godey.
\newblock A bifurcation analysis for the {L}ugiato-{L}efever equation.
\newblock {\em The European Physical Journal D}, 71:1--17, 2017.

\bibitem{Godey2014Stability}
C.~Godey, I.~V. Balakireva, A.~Coillet, and Y.~K. Chembo.
\newblock Stability analysis of the spatiotemporal {L}ugiato-{L}efever model
  for kerr optical frequency combs in the anomalous and normal dispersion
  regimes.
\newblock {\em Physical Review A}, 89(6):063814, 2014.

\bibitem{Hakkaev2019Generation}
S.~Hakkaev, M.~Stanislavova, and A.~G. Stefanov.
\newblock On the generation of stable {K}err frequency combs in the
  {L}ugiato-{L}efever model of periodic optical waveguides.
\newblock {\em SIAM J. Appl. Math.}, 79(2):477--505, 2019.

\bibitem{Haragus2021Linear}
M.~Haragus, M.~A. Johnson, and W.~R. Perkins.
\newblock Linear modulational and subharmonic dynamics of spectrally stable
  {L}ugiato-{L}efever periodic waves.
\newblock {\em J. Differential Equations}, 280:315--354, 2021.

\bibitem{Haragus2023Nonlinear}
M.~Haragus, M.~A. Johnson, W.~R. Perkins, and B.~de~Rijk.
\newblock Nonlinear modulational dynamics of spectrally stable
  {L}ugiato-{L}efever periodic waves.
\newblock {\em Ann. Inst. H. Poincar\'{e} C Anal. Non Lin\'{e}aire},
  40(4):769--802, 2023.

\bibitem{Haragus2024Nonlinear}
M.~Haragus, M.~A. Johnson, W.~R. Perkins, and B.~de~Rijk.
\newblock Nonlinear subharmonic dynamics of spectrally stable
  {L}ugiato-{L}efever periodic waves.
\newblock {\em Comm. Math. Phys.}, 405(10):Paper No. 227, 2024.

\bibitem{Haerterich1998Cascades}
J.~H\"{a}rterich.
\newblock Cascades of reversible homoclinic orbits to a saddle-focus
  equilibrium.
\newblock {\em Phys. D}, 112(1-2):187--200, 1998.

\bibitem{Herr2014Temporal}
T.~Herr, V.~Brasch, J.~D. Jost, C.~Y. Wang, N.~M. Kondratiev, M.~L. Gorodetsky,
  and T.~J. Kippenberg.
\newblock Temporal solitons in optical microresonators.
\newblock {\em Nature Photonics}, 8(2):145--152, 2014.

\bibitem{Homburg2010Homoclinic}
A.~J. Homburg and B.~Sandstede.
\newblock Homoclinic and heteroclinic bifurcations in vector fields.
\newblock In {\em Handbook of Dynamical Systems}, volume~3, pages 379--524.
  Elsevier, 2010.

\bibitem{Kato1995Perturbation}
T.~Kato.
\newblock {\em Perturbation theory for linear operators}.
\newblock Classics in Mathematics. Springer-Verlag, Berlin, 1995.
\newblock Reprint of the 1980 edition.

\bibitem{Lin1990Melnikovs}
X.-B. Lin.
\newblock Using {M}elnikov's method to solve \v{S}ilnikov's problems.
\newblock {\em Proc. Roy. Soc. Edinburgh Sect. A}, 116(3-4):295--325, 1990.

\bibitem{Lugiato1987Spatial}
L.~A. Lugiato and R.~Lefever.
\newblock Spatial dissipative structures in passive optical systems.
\newblock {\em Phys. Rev. Lett.}, 58(21):2209--2211, 1987.

\bibitem{Mandel2017Apriori}
R.~Mandel and W.~Reichel.
\newblock A priori bounds and global bifurcation results for frequency combs
  modeled by the {L}ugiato-{L}efever equation.
\newblock {\em SIAM J. Appl. Math.}, 77(1):315--345, 2017.

\bibitem{Marin-Palomo2017}
P.~Marin-Palomo, J.~N. Kemal, M.~Karpov, A.~Kordts, J.~Pfeifle, M.~H.~P.
  Pfeiffer, P.~Trocha, S.~Wolf, V.~Brasch, M.~H. Anderson, R.~Rosenberger,
  K.~Vijayan, W.~Freude, T.~J. Kippenberg, and C.~Koos.
\newblock Microresonator-based solitons for massively parallel coherent optical
  communications.
\newblock {\em Nature}, 546(7657):274--279, Jun 2017.

\bibitem{Miyaji2010Bifurcation}
T.~Miyaji, I.~Ohnishi, and Y.~Tsutsumi.
\newblock Bifurcation analysis to the {L}ugiato-{L}efever equation in one space
  dimension.
\newblock {\em Phys. D}, 239(23-24):2066--2083, 2010.

\bibitem{Miyaji2011Stability}
T.~Miyaji, I.~Ohnishi, and Y.~Tsutsumi.
\newblock Stability of a stationary solution for the {L}ugiato-{L}efever
  equation.
\newblock {\em Tohoku Math. J. (2)}, 63(4):651--663, 2011.

\bibitem{Palmer1984Exponential}
K.~J. Palmer.
\newblock Exponential dichotomies and transversal homoclinic points.
\newblock {\em J. Differential Equations}, 55(2):225--256, 1984.

\bibitem{ParraRivas2018BifurcationLocalized}
P.~Parra-Rivas, D.~Gomila, L.~Gelens, and E.~Knobloch.
\newblock Bifurcation structure of localized states in the {L}ugiato-{L}efever
  equation with anomalous dispersion.
\newblock {\em Phys. Rev. E}, 97(4):042204, 20, 2018.

\bibitem{Picque2019}
N.~Picqu{\'e} and T.~W. H{\"a}nsch.
\newblock Frequency comb spectroscopy.
\newblock {\em Nature Photonics}, 13(3):146--157, Mar 2019.

\bibitem{Rowlands1974Stability}
G.~Rowlands.
\newblock {On the stability of solutions of the non-linear {S}chr\"odinger
  equation}.
\newblock {\em IMA Journal of Applied Mathematics}, 13(3):367--377, 06 1974.

\bibitem{Sandstede1993Verzweigungstheorie}
B.~Sandstede.
\newblock {\em Verzweigungstheorie homokliner {V}erdopplungen}.
\newblock PhD thesis, University of Stuttgart, 1993.

\bibitem{Sandstede1998Stability}
B.~Sandstede.
\newblock Stability of multiple-pulse solutions.
\newblock {\em Trans. Amer. Math. Soc.}, 350(2):429--472, 1998.

\bibitem{Sandstede1999Instability}
B.~Sandstede.
\newblock Instability of localized buckling modes in a one-dimensional strut
  model [{MR}1484137 (98g:73021)].
\newblock In {\em Localization and solitary waves in solid mechanics},
  volume~12 of {\em Adv. Ser. Nonlinear Dynam.}, pages 54--68. World Sci.
  Publ., River Edge, NJ, 1999.

\bibitem{Sandstede1997Fibers}
B.~Sandstede, C.~K. R.~T. Jones, and J.~C. Alexander.
\newblock Existence and stability of {$N$}-pulses on optical fibers with
  phase-sensitive amplifiers.
\newblock {\em Phys. D}, 106(1-2):167--206, 1997.

\bibitem{Sandstede2001Stability}
B.~Sandstede and A.~Scheel.
\newblock On the stability of periodic travelling waves with large spatial
  period.
\newblock {\em J. Differential Equations}, 172(1):134--188, 2001.

\bibitem{Sandstede2012}
B.~Sandstede, A.~Scheel, G.~Schneider, and H.~Uecker.
\newblock Diffusive mixing of periodic wave trains in reaction-diffusion
  systems.
\newblock {\em J. Differential Equations}, 252(5):3541--3574, 2012.

\bibitem{Schliesser2005}
A.~Schliesser, M.~Brehm, F.~Keilmann, and D.~W. van~der Weide.
\newblock Frequency-comb infrared spectrometer for rapid, remote chemical
  sensing.
\newblock {\em Opt. Express}, 13(22):9029--9038, Oct 2005.

\bibitem{Stanislavova2018Asymptotic}
M.~Stanislavova and A.~G. Stefanov.
\newblock Asymptotic stability for spectrally stable {L}ugiato-{L}efever
  solitons in periodic waveguides.
\newblock {\em J. Math. Phys.}, 59(10):101502, 12, 2018.

\bibitem{Sun2018}
C.~Sun, T.~Askham, and J.~N. Kutz.
\newblock Stability and dynamics of microring combs: elliptic function
  solutions of the {L}ugiato-{L}efever equation.
\newblock {\em J. Opt. Soc. Am. B}, 35(6):1341--1353, 2018.

\bibitem{Udem2002}
T.~Udem, R.~Holzwarth, and T.~W. H{\"a}nsch.
\newblock Optical frequency metrology.
\newblock {\em Nature}, 416(6877):233--237, Mar 2002.

\bibitem{Uecker2014pde2path}
H.~Uecker, D.~Wetzel, and J.~D.~M. Rademacher.
\newblock pde2path---a {M}atlab package for continuation and bifurcation in
  2{D} elliptic systems.
\newblock {\em Numer. Math. Theory Methods Appl.}, 7(1):58--106, 2014.

\bibitem{Vanderbauwhede1992Homoclinic}
A.~Vanderbauwhede and B.~Fiedler.
\newblock Homoclinic period blow-up in reversible and conservative systems.
\newblock {\em Z. Angew. Math. Phys.}, 43(2):292--318, 1992.

\bibitem{Weiner2017}
A.~M. Weiner.
\newblock Cavity solitons come of age.
\newblock {\em Nature Photonics}, 11(9):533--535, Sep 2017.

\bibitem{Weinstein1985Modulational}
M.~I. Weinstein.
\newblock Modulational stability of ground states of nonlinear
  {S}chr\"{o}dinger equations.
\newblock {\em SIAM J. Math. Anal.}, 16(3):472--491, 1985.

\bibitem{Weinstein1986Lyapunov}
M.~I. Weinstein.
\newblock Lyapunov stability of ground states of nonlinear dispersive evolution
  equations.
\newblock {\em Comm. Pure Appl. Math.}, 39(1):51--67, 1986.

\bibitem{Zumbrun2024Forward}
K.~Zumbrun.
\newblock Forward-modulated damping estimates and nonlocalized stability of
  periodic {L}ugiato-{L}efever waves.
\newblock {\em Ann. Inst. H. Poincar\'e{} C Anal. Non Lin\'eaire},
  41(2):497--510, 2024.

\end{thebibliography}

\end{document}